\documentclass[preprint,11pt,authoryear]{elsarticle}

\usepackage[utf8]{inputenc}
\usepackage{geometry}
\usepackage[utf8]{inputenc}
\usepackage{graphicx} 
\usepackage{bussproofs}
\usepackage{amsmath,amssymb,amsfonts}
\usepackage{mathrsfs,latexsym,amsthm,enumerate}
\usepackage{tikz}
\usepackage{tikz-cd}
\usetikzlibrary{cd}
\usepackage{natbib}
\usepackage{amscd}
\usepackage[all]{xy}
\usepackage[utf8]{inputenc}
\usepackage{amsmath,amssymb,amsfonts}
\usepackage{mathrsfs,latexsym,amsthm,enumerate}
\usepackage{breqn}
\usepackage{fancyhdr}
\usepackage{sectsty}
\usepackage{stmaryrd}
\usepackage{url}
\usepackage{bbm}
\usepackage{subcaption}
\usepackage{float}
\catcode`@=11
\def\caseswithdelim#1#2{\left#1\,\vcenter{\normalbaselines\m@th
  \ialign{\strut$##\hfil$&\quad##\hfil\crcr#2\crcr}}\right.}
\catcode`@=12

\allsectionsfont{\sffamily\mdseries\upshape}

\catcode`@=11
\def\caseswithdelim#1#2{\left#1\,\vcenter{\normalbaselines\m@th
  \ialign{\strut$##\hfil$&\quad##\hfil\crcr#2\crcr}}\right.}
\catcode`@=12

\allsectionsfont{\sffamily\mdseries\upshape}

\theoremstyle{plain}
\newtheorem{lemma}{Lemma}[section]
\newtheorem{proposition}[lemma]{Proposition}
\newtheorem{corollary}[lemma]{Corollary}

\newtheorem{theorem}[lemma]{Theorem}
\newtheorem{observation}[lemma]{Observation}

\theoremstyle{definition}

\newcommand{\lra}{\leftrightarrows}
\newcommand{\ra}{\rightarrow}
\newcommand{\epi}{\twoheadrightarrow}
\newcommand{\inclu}{\hookrightarrow}

\newcommand{\up}{{\uparrow}}

\newcommand{\tc}{\textit}
\newcommand{\mb}[1]{\mbox{#1}}
\newcommand{\ca}{\mathcal}
\newcommand{\mf}{\mathsf}

\newcommand{\se}{\subseteq}
\newcommand{\sm}{\setminus}

\newcommand{\we}{\wedge}
\newcommand{\ve}{\vee}
\newcommand{\bwe}{\bigwedge}
\newcommand{\bve}{\bigvee}

\newcommand{\bca}{\bigcap}

\newcommand{\bd}[1]{\mathbf{#1}}

\newcommand{\bl}{\mathfrak{b}}
\newcommand{\op}{\mathfrak{o}}
\newcommand{\cl}{\mathfrak{c}}

\newcommand{\spa}{\mathsf{sp}}
\newcommand{\sob}{\mathsf{sob}}

\newcommand{\id}{\mathsf{id}}

\newcommand{\Om}{\Omega}

\newcommand{\pt}{\mathsf{pt}}

\newcommand{\SL}{\mathsf{S}(L)}
\newcommand{\SLop}{\mathsf{S}(L)^{op}}

\newcommand{\SO}{\mathsf{S}(\Omega (X))}
\newcommand{\SOop}{\mathsf{S}(\Omega (X))^{op}}

\title{Revisiting the relation between subspaces and sublocales}
\author{Anna Laura Suarez\fnref{address}}
\address{CMUC, Department of Mathematics, University of Coimbra, 3001-454, Coimbra, Portugal}
\address{School of Computer Science, University of Birmingham, B15 2TT, Birmingham, UK}
\ead{axs1431@cs.bham.ac.uk}

\begin{document}

\begin{frontmatter}

\begin{abstract}
We revisit results concerning the connection between subspaces of a space and sublocales of its locale of open sets. The approach we present is based on the notion of sublocale as a concrete subcollection of a locale. We characterize the frames $L$ such that the spatial sublocales of $\mathsf{S}(L)$ perfectly represent the subspaces of $\mathsf{pt}(L)$. We prove choice-free, weak versions of the results by Niefield and Rosenthal characterizing those frames such that all their sublocales are spatial. We do so by using a notion of essential prime which does not rely on the existence of enough minimal primes above every element. We will re-prove Simmons' result that spaces such that the sublocales of $\Om (X)$ perfectly represent their subspaces are exactly the scattered spaces. We will characterize scattered spaces in terms of a strong form of essentiality for primes. We apply these characterizations to show that, when $L$ is a spatial frame and a coframe, $\pt(L)$ is scattered if and only if it is $T_D$, and this holds if and only if all the primes of $L$ are completely prime.
\end{abstract}


\end{frontmatter}

\tableofcontents

\section{Introduction}

One of the features distinguishing pointfree topology from point-set topology is the behaviour of subspaces. While for a space $X$ the lattice of its subspaces is simply the powerset $\ca{P}(X)$ of its underlying set, the lattice $\SL$ of sublocales of a locale $L$ is in general just a coframe. In particular, the lattice $\SO$ of sublocales of the locale of open sets $\Om (X)$ of a space $X$ need not represent perfectly the subspaces of $X$, and might have in general many nonspatial sublocales. An account of the theory of locales and sublocales can be found in \cite{johnstone82,picadopultr2011frames, isbell72}.

\begin{enumerate}

     \item A question addressed in this article is what are those frames $L$ such that $\ca{P}(\pt(L))$ is isomorphic to the collection of all \textit{spatial} sublocales in $\mf{S}(L)$. These are the frames for which $\pt(L)$ is a $T_D$ space. We provide several characterizations of these frames and, in particular, we show that the condition is equivalent to the spatialization of $\mf{S}(L)^{op}$ being Boolean, as well as all of the primes in $L$ being \textit{weakly covered}.
     
     \item A related question that has been tackled in the literature is when is it that $\SLop$ is spatial. This question is explored in \cite{niefield87} and \cite{ISBELL91}. In \cite{niefield87} it is shown that $\SLop$ is spatial when its prime elements satisfy a certain condition; and in \cite{ISBELL91} it is shown that $\SOop$ is spatial exactly when the spectrum $\pt(\Om (X))$ is such that every closed subspace has a discrete dense subspace.
     
    \item The question of when $\ca{P}(X)$ is isomorphic to $\SO$ has been explored in \cite{simmons80}, where spaces with this property are characterized. In particular it is proven that these are exactly the \textit{scattered} spaces. We will explore the similar question of characterizing those frames $L$ such that $\mf{S}(L)$ is isomorphic to $\ca{P}(\pt(L))$.

\end{enumerate}

The work presented here is essentially an analysis of the condition described in point (1), and a revisitation of the existing results on conditions described in (2) and (3), which we carry out as pointfreely as possible, avoiding calculations inside point-set spaces. 

Our starting point is the observation that for any frame $L$ there is an order isomorphism between the collection of \textit{sobrifications} of the subspaces of $\pt(L)$ (ordered under subspace inclusion) and that of the \textit{spatializations} of the sublocales of $L$.  
Both these collections are coframes which we denote as $\sob [\ca{P}(\pt(L))]$ and $\spa[\mf{S}(L)]$, respectively. We will show that these two coframes are naturally seen as part of a commuting diagram in the category of coframes of the following form:

\begin{center}
\begin{tikzcd}[row sep=large, column sep = large]
\spa [\mf{S}(L)]  
\arrow{r}{\pt(\cong)}  
& \sob[\ca{P}(\pt(L))] 
\arrow[hookrightarrow]{d}  \\
\mf{S}(L) 
\arrow{r}{\pt} 
\arrow[twoheadrightarrow]{u}{\spa} 
& \ca{P}(\pt(L))  
\end{tikzcd}
\end{center}

 Our investigation will proceed as an exploration of the three conditions above, keeping in mind that condition (1) amounts to the right vertical arrow being an isomorphism, that condition (2) amounts to the left vertical arrow being an isomorphism, and that condition (3) amounts to both of them being isomorphisms. 

On one hand the aim of this article is to gather results already present in various places in the literature in a cohesive narrative. On the other hand, we will also prove new results that stem naturally from our approach, in particular from adding condition (1) to the picture, and from analyzing all three conditions in terms of prime elements. We show that frames $L$ satisfying condition (1) are those such that the primes of $L$ are all weakly covered. Coveredness is studied in \cite{banaschewski15} and \cite{picado19}. We will prove choice-free versions of the results in \cite{niefield87} characterizing frames satisfying condition (2). In particular, we use a notion of essential prime that does not assume the existence of enough minimal primes above an element. We will characterize frames satisfying condition (3) as those frames $L$ such that for every element in $L$ there is a prime above it satisfying a strong version of essentiality, which we call \tc{absolute essentiality}. 

\section{Preliminaries}

We assume familiarity with the basic notions of pointfree topology. Nonetheless, let us give reminders of what specifically are the basic notions that we will use without mention in this paper.

\begin{itemize}
    \item The adjunction $\Om:\mathbf{Top}\lra \mathbf{Frm}^{op}:\pt$ with $\Om \dashv \pt$ connects the categories of frames with that of topological spaces. The functor $\Om$ assigns to each space its lattice of opens, and $\pt$ assigns to a frame $L$ the collection of the frame maps $f:L\ra 2$, topologized by setting the opens to be exactly the sets of the form $\{f: L\ra 2: f(a)=1\}$ for some $a\in L$. 
    \item A frame $L$ is \textit{spatial} if for $a,b\in L$ whenever $a\nleq b$ there is some frame map $f: L \ra 2$ such that $f(a)=1\neq f(b)$. Spatial frames are exactly those of the form $\Om (X)$ for some space $X$.
    \item A space is \textit{sober} if every irreducible closed set is the closure of a unique point. Sober spaces are exactly those of the form $\pt (L)$ for some frame $L$.
    \item The adjunction $\Om \dashv \pt$ restricts to a dual equivalence of categories between spatial frames and sober spaces.
    \item The category of sober spaces is a full reflective subcategory of $\mathbf{Top}$. For each space $X$ we have a \textit{sobrification} map $N: X\ra \pt (\Om (X))$ mapping each point $x\in X$ to the map $f_x: \Om (X)\ra 2$ defined as $f(U)=1$ if and only if $x\in U$. 
    \item The category of spatial frames is a full reflective subcategory of $\mathbf{Frm}$. For each frame we have a \textit{spatialization} map $\phi: L \ra \Om (\pt (L))$ which sends each $a\in L$ to $\{f: L\ra 2: f(a)=1\}$.
    \item We call $\mathbf{Loc}$ the category $\mathbf{Frm}^{op}$, and we call its
    objects \textit{locales}. Maps in the category of locales have a concrete description: they can be characterized as the right adjoints of frame maps (since frame maps preserve all joins, they always have right adjoints). 
    \item A \textit{sublocale} of a locale $L$ is a subset $S\se L$ such that it is closed under arbitrary meets, and such that $s\in S$ implies $x\ra s\in S$ for every $x\in L$. This is equivalent to $S\se L$ being a locale in the inherited order, and the subset inclusion being a map in $\mathbf{Loc}$.
    \item Sublocales of $L$ are closed under arbitrary intersections, and so the collection $\SL$ of all sublocales of $L$, ordered under set inclusion, is a complete lattice. Joins in $\SL$ are computed as $\bigvee_i S_i=\{\bigwedge M: M\se \bigcup_i S_i\}$.
    \item In the coframe $\SL$ the bottom element is the sublocale $\{1\}$ and the top element is $L$.
    \item Embedded in $\SL$ we have the coframe of \textit{closed sublocales} which is isomorphic to $L^{op}$. The closed sublocale $\cl (a)\se L$ is defined to be ${\uparrow}a$ for $a\in L$.
    \item Embedded in $\SL$ we also have the frame of \textit{open sublocales} which is isomorphic to $L$. The open sublocale $\op (a)$ is defined to be $\{a\ra x: x\in L\}$ for $a\in L$.
    \item The sublocales $\op (a)$ and $\cl (a)$ are complements of one another in the coframe $\SL$ for any element $a\in L$.
    \item Furthermore, open and closed sublocales generate the coframe $\SL$ in the sense that for each $S\in \SL$ we have $S=\bigcap \{\op (x)\vee \cl (y): S\se \op (x)\vee \cl (y)\}$.
    
\end{itemize}

\subsection{The axiom $T_D$}

 A space $X$ is defined to be $T_D$ if every point $x\in X$ has some neighborhood $U$ with $U{\sm}\{x\}$ open. This is an axiom stronger than $T_0$ and weaker than $T_1$ and it was introduced in \cite{Aull62}. It has been used (for instance in \cite{Pultr94}) in order to answer the question of when a topological space can be completely recovered from its frame of opens. The \textit{Skula space} of a space $X$, denoted as $Sk(X)$, is the space defined as follows. The set of points is the same set of points as $X$, while the topology is the \textit{Skula topology}, that is the one generated by the opens of $X$ together with their complements. The following characterizations of $T_D$ spaces can be found in \cite{picadopultr2011frames}.
 
 \begin{proposition}
 The following are equivalent for a $T_0$ space $X$.
 \begin{enumerate}
     \item The space is $T_D$.
     \item For no $x\in X$ do we have that the dualization of the inclusion $X{\sm}\{x\}\se X$ is a frame isomorphism.
     \item If $Y\nsubseteq Z$ then $\Om'(Y)\nsubseteq \Om'(Z)$, for all subspaces $Y,Z\se X$.
     \item The Skula space $Sk(X)$ is discrete.

 \end{enumerate}
 \end{proposition}

One can exhibit the role of the axiom $T_D$ in pointfree topology by comparing it to sobriety. This has been done in detail in \cite{banaschewskitd}. In particular, in the context of $T_0$ spaces, the two axioms mirror each other in that sober spaces are maximal in the same sense in which $T_D$ spaces are minimal.

\begin{itemize}
    \item A space $X$ is sober if and only if we can never have a nontrivial subspace inclusion $X\inclu Y$ such that its dualization is an isomorphism.
    \item A space $X$ is $T_D$ if and only if we can never have a nontrivial subspace inclusion $Y\inclu X$ such that its dualization is an isomorphism. 
\end{itemize}
Comparing the two axioms with respect to their implications for subspaces and sublocales yields the following.
\begin{itemize}
    \item A space $X$ is sober if and only all spatial sublocales of $\Om (X)$ are induced by subspaces of $X$.
    \item A space $X$ is $T_D$ if and only if different subspaces of $X$ induce different spatial sublocales.
\end{itemize}

\section{The coframe of spatializations}

We will look at spatialization as an interior on the coframe $\SL$ of sublocales of a frame $L$. Since we will work on the coframe $\SL$ of sublocales of a frame $L$, we will adapt the terminology from frame theory by just adding the prefix ``co-" to the classical notion. For example, for a coframe $C$ we will call ``conucleus" an operator $\nu:C\ra C$ which is a nucleus on $C^{op}$, ``cosublocale" a subset $S\se C$ which is a sublocale of $C^{op}$, and so on. We also use the following pieces of notation.
\begin{itemize}
    \item For a frame $L$, we call $\pt(L)$ its prime spectrum.
    \item For a complete lattice $L$ and a subset $Y\se L$, we call $\mathfrak{M}(Y)$ the ordered collection $\{\bwe M:M\se Y\}$. Similarly, we call $\mathfrak{m}(Y)$ the collection $\{\bwe M:M\se Y,M\mb{ is finite}\}$.
\end{itemize}

The first step is to notice that the spatialization $\phi_L:L\epi \Om(\pt(L))$ of a frame $L$, as any frame surjection, determines a sublocale of $L$. We now show how to obtain this sublocale explicitly. Before doing this, we will gather some useful facts that will simplify our calculations with sublocales. First of all, let us remind the reader that for every frame $L$ and every element $x\in L$, there exists a sublocale $\bl(x)$ defined as $\bca \{S\in \SL:x\in S\}$. This is clearly the smallest sublocale containing the element $x$. The reason behind this notation is that ``$\bl$" stands for ``Boolean" - the sublocales of $L$ that are Boolean algebras coincide with those of the form $\bl(x)$ for some $y\in L$. This is not a fact that we will use in our analysis, however. We provide proofs for the facts on primes and Boolean sublocales that we show, but proofs of these results are also in \cite{picadopultr2011frames}. Before we start with basic facts about Boolean sublocales, let us look at very general properties of the Heyting implication in frames. 
\begin{lemma}\label{impliadj}
For any frame $L$ we have a contravariant Galois connection
\[
-\ra x:L\lra L:-\ra x
\]
\end{lemma}
\begin{proof}
Let us fix an element $x\in L$. For any two elements $a,b\in L$ we have that $a\leq b\ra x$ if and only if $a\we b\leq x$ and this holds if and only if $b\leq a\ra x$. Then, indeed the pair of maps described in the statement form a Galois connection.
\end{proof}
From the properties of antitone Galois connections, the following facts hold.
\begin{corollary}\label{implijoinmeet}
For any frame $L$ and every $x\in L$ we have the following.
\begin{enumerate}
    \item $\bve_i y_i\ra x=\bwe(y_i\ra x)$ for every collection $y_i\in L$.
    \item $y\leq (y\ra x)\ra x$ for every $y\in L$
    \item $y\ra x=((y\ra x)\ra x)\ra x$ for every $y\in L$.
\end{enumerate}
\end{corollary}
\begin{proof}
Let us prove each item as a chain of equivalences.
\begin{enumerate}
    \item Suppose that $y_i\in L$ is any collection. We have the following equivalences for each $z\in L$.
    \begin{prooftree}
     \AxiomC{$z\leq \bwe(y_i\ra x)$}
     \UnaryInfC{$z\leq (y_i\ra x)$}
     \UnaryInfC{$z\we y_i\leq x$}
     \UnaryInfC{$\bve_i z\we y_i\leq x$}
     \UnaryInfC{$z\we \bve_i y_i\leq x$}
     \UnaryInfC{$z\leq \bve_i y_i\ra x$}
     \end{prooftree}
    \item Suppose that $y\in L$. The following are equivalent.
      \begin{prooftree}
     \AxiomC{$y\we x\leq x$}
     \UnaryInfC{$y\we (y\ra x)\leq x$}
     \UnaryInfC{$y\leq (y\ra x)\ra x$}
     \end{prooftree}
   \item Suppose that $y\in L$. One inequality follows from item (2). For the other, by the antitone Galois connection in Lemma \ref{impliadj}, we have that these are equivalent.
     \begin{prooftree}
     \AxiomC{$((y\ra x)\ra x)\ra x\leq y\ra x$}
     \UnaryInfC{$y\leq (((y\ra x)\ra x)\ra x)\ra x$}
     \end{prooftree}
     The bottom line is always true by item (2).\qedhere
\end{enumerate}
\end{proof}
 Let us see some results on Boolean sublocales. Compare the two results in the lemma below with the following two facts about open sublocales.
 \begin{enumerate}
     \item For every $x\in L$ we have that $\op (x)=\{x\ra y:y\in L\}$
     \item For every $x\in L$ we have that $a\in \op (x)$ if and only if $a=x\ra a$.
 \end{enumerate}
\begin{lemma}\label{booleanfacts}
Let $L$ be a frame. We have the following facts.
\begin{enumerate}
\item For every $x\in L$ we have $\bl(x)=\{y\ra x:y\in L\}$.
\item For every $x\in L$ we have that $a\in \bl(x)$ if and only if $a=(a\ra x)\ra x$.
\item For every $x\in L$ the frame $\bl(x)$ is a Boolean algebra, where the negation of an element $a\in \bl(x)$ is $a\ra x$.
\end{enumerate}
\end{lemma}
\begin{proof}
Let us prove each item in turn.
\begin{enumerate}
    \item By definition of $\bl(x)$, it suffices to show that $\{y\ra x:y\in L\}$ is a sublocale containing $x$, and that it is the smallest such sublocale. We have that $1\ra x=x$, and so $x\in \{y\ra x:y\in L\}$. Let us show that this collection is a sublocale. For every $a\in L$, we have that $a\ra (y\ra x)=(a\we y)\ra x$, and so the collection is stable under left implication. For closure under meets, we notice that by item (1) of Corollary \ref{implijoinmeet} we have that $\bwe (y_i\ra x)=\bve_i y_i\ra x$. Finally, by the stability property of sublocales, whenever $x\in S$ for some sublocale $S\se L$, we must have $\{y\ra x:y\in L\}\se S$.
    \item If $a=(a\ra x)\ra x$ then $a\in \bl(x)$ by item (1). Suppose, now, that $a\in \bl(x)$. We have that $a\leq (a\ra x)\ra x$ by item (2) of Corollary \ref{implijoinmeet}. For the other inequality, let us use the assumption that $a\in \bl(x)$. In particular, let $y\in L$ be such that $a=y\ra x$. By item (3) of Corollary \ref{implijoinmeet}, we have $a=y\ra x=((y\ra x)\ra x)\ra x$.
    \item The collection $\bl(x)$ is a frame. To show that it is a Boolean algebra, it suffices to show that each element joins to $1$ with its Heyting pseudocomplement. First, let us show that the Heyting pseudocomplement of some $a\in \bl(x)$ is $a\ra x$. The bottom element of the sublocale $\bl(x)$ is $x$ itself, by item (1) and because $x\leq y\ra x$ for each $y\in L$; furthermore $x\in \bl(x)$ as $1\ra x=x$. For every $a\in \bl(x)$, we have that $y\leq a\ra x$ if and only if $y\we x\leq a$. Since meets in $\bl(x)$ are computed in the same way as in $L$, this means that the element $x\ra a$ satisfies the definition of pseudocomplement of $x$ in $\bl(x)$. Item (2), then, implies that every element of $\bl(x)$ is its own Heyting pseudocomplement, and so $\bl(x)$ is a Boolean algebra.\qedhere
\end{enumerate}
\end{proof}
Let us see some more generic locale theoretical facts which we will use to simplify our calculations.
\begin{lemma}\label{manylocalethings}
For a frame $L$ the following facts hold.
\begin{enumerate}
    \item For any collection $\{S_1,...,S_n\}\se L$ of sublocales, we have that 
    \[
    S_1\ve ...\ve S_n=\mathfrak{m}(S_1\cup ...\cup S_n).
    \]
   \item If $p\in \pt(L)$, then for each $x\in L$ we have that $x\ra p=p$ if and only if $x\nleq p$.  
   \item If $p\in L$ is such that $y\nleq p$ implies $y\ra p=p$ for every $y\in L$ then $p$ must be prime.
   \item For every prime $p\in \pt(L)$, we have that $\bl(p)=\{1,p\}$.
   \item If a sublocale $S$ has two elements, it is of the form $\bl(p)$ for some prime $p\in \pt(L)$.
 
 \item For a collection $p_i\in \pt(L)$ we have $\bve_i \bl(p_i)=\mathfrak{M}(\{p_i:i\in I\})$.
\end{enumerate}
\end{lemma}
\begin{proof}
Let us prove each item in turn.
\begin{enumerate}
\item We will show that $\mathfrak{m}(S_1\cup ...\cup S_n)$ is a sublocale, and that it is the smallest one which contains each $S_m$. Since each of the $S_m$'s is closed under meets, the collection $\mathfrak{m}(S_1\cup ...\cup S_n)$ is closed under all meets, too. If $y\in L$, let us consider an arbitrary element $s_1\we ...\we s_n\in \mathfrak{m}(S_1\cup ...\cup S_n)$; say that $s_m\in S_m$ for each $m\leq n$, with some $s_m$'s possibly equal to $1$. The element $y\ra (s_1\we ...\we s_n)$ is the same as $\bwe_{m\leq n}(y\ra s_m)$. Since every sublocale $S_m$ is stable under left implication, indeed we have $y\ra s_m\in S_m$ and so $y\ra (s_1\we ...\we _n)\in \mathfrak{m}(S_1\cup ...\cup S_n)$. This set is then a sublocale, and it clearly contains all of the $S_m$'s. For $T\se L$ a sublocale, we have that $S_1,...,S_n\se T$ implies that $\mathfrak{m}(S_1\cup ...\cup S_n)\se T$, by meet closure of $T$. 
\item Suppose that $p\in L$ is a prime element and that $x\nleq p$. We expand the definition of $x\ra p$ to obtain that this is the same as $\bve \{y\in L:y\leq p\mb{ or }x\leq p\}$, in which expression we have used primality of $p$. Then, $x\ra p=p$. Conversely, if for $x\in L$ we have that $x\ra p=p$ for a prime $p$, we deduce that $x\ra p\neq 1$, since primes are different from $1$ by definition. Then, we cannot have $x\leq p$. 

\item Suppose that $p\in L$ is such that $x\nleq p$ implies that $x\ra p=p$ for every $x\in L$. To show that $p$ is prime, we show that if we have $x\we y\leq p$, and $x\nleq p$, this implies $y\leq p$. Then, let us assume the antecedent. Since $y\nleq p$, by assumption $y\ra p=p$. This means, by definition of Heyting implication, that $y\we z\leq p$ implies that $z\leq p$. In particular, we deduce that $x\leq p$. So, $p$ is prime. 
\item First, let us show that for every prime $p$ we have that $\{1,p\}$ is a sublocale. This set is immediately seen to be closed under meets. For the stability condition, we notice that for $y\in L$ we have that $y\ra 1=1\in \{1,p\}$ for every $y\in L$; and that we also have $y\ra p=1\in \{1,p\}$ if $y\leq p$, and $y\ra p=p\in \{1,p\}$ if $y\nleq p$, by item (2). Then, indeed, $\{1,p\}$ is a sublocale. It is clearly the smallest sublocale containing $p$: since sublocales are closed under arbitrary meet, in particular they must all contain the empty meet $1$.
\item Suppose that for some element $x\in L$ we have that $\{1,p\}$ is a sublocale, where $1\neq p$. If $y\nleq p$ we must have that $y\ra p\neq 1$ and that $y\ra p\in \{1,p\}$, by stability under left implication. Then, $y\nleq p$ implies that $y\ra p=p$ for all $y\in L$; by item (3) this implies that $p$ is prime.

\item First, we show that $\mathfrak{M}(\{p_i:i\in I\})$ is a sublocale for every collection of primes $\{p_i:i\in I\}$. By its definition, the collection is closed under all meets. Suppose that $\bwe_ip_j$ is an arbitrary element of our collection, with each $p_j$ a prime in $\{p_i:i\in I\}$. We have that $y\ra \bwe_j p_j=\bwe_j(y\ra p_j)$. For every $j$, by item (2) we have that each $y\ra p_j$ is either $1$ or $p_j$, and so $y\ra \bwe_j p_j=\bwe_kp_k$, with each $p_k$ equal to some $p_j$. By the definition of this collection, this is in $\mathfrak{M}(\{p_i:i\in I\})$. \qedhere
\end{enumerate}
\end{proof}

The next results will be tools to define a spatialization operator on $\SL$.
\begin{lemma}
For any frame $L$ and any sublocale $S\se L$, we have that $\pt(S)=\pt(L)\cap S$.
\end{lemma}
\begin{proof}
For the inclusion $\pt(L)\cap S\se \pt(S)$, let us assume that $p\in \pt(L)$ is a prime which is in $S$. The element is also prime as an element of $S$ since by definition this condition is formally stronger than $p\in S$ and $p$ being prime. Conversely, if $p\in S$ is prime as an element of $S$, suppose that $x\we y\leq p$ for $x,y\in L$. Then, by monotonicity of nuclei we have that $\nu_S(x\we y)\leq \nu_S(p)=p$, and by their meet preservation property also $\nu_S(x)\we \nu_S(y)\leq p$. By primality of $p$ in $S$, we have that either $\nu_S(x)\leq p$ or $\nu_S(y)\leq p$. Since nuclei are inflationary, if the former holds then $x\leq p$, and if the latter holds $y\leq p$. Then, $p$ is prime as an element of $L$, too. 
\end{proof}
\begin{lemma}
A frame is spatial if and only if all its elements are meets of primes.
\end{lemma}
\begin{proof}
By definition of the prime spectrum of a frame, we know that a frame is spatial if and only if the map $\phi'_L:L\epi \Om(\pt(L))$ is injective, that is, if whenever $a\nleq b$ we have some prime $p\in L$ such that $b\leq p$ and $a\nleq p$. If all elements of $L$ are meets of primes, each element must be the meet of all the primes above it. then whenever $a\nleq b$, there must be a prime $p$ with $b\leq p$ and $a\nleq p$, otherwise we would have that $\up b\cap \pt(L)\se \up a\cap \pt(L)$, from which $\bwe \{p\in \pt(L):a\leq p\}\leq \bwe \{p\in \pt(L):b\leq p\}$, contradicting our hypothesis that $s\nleq b$. Conversely, if the frame $L$ is spatial, and if we have some $a\in L$ such that $a\neq \bwe _ip_i$ for every collection $p_i\in \pt(L)$, we also have that $a<\bwe \{p\in \pt(L):a\leq p\}=:a'$. But since these two elements of $L$ have the same primes above them, we deduce that $\phi_L'(a)=\phi_L'(a')$, contradicting spatiality of $L$.
\end{proof}
\begin{lemma}\label{spaint}
For any frame $L$, its spatialization sublocale is $\mathfrak{M}(\pt(L))$.
\end{lemma}
\begin{proof}
Translating the universal property of the frame spatialization into a condition on sublocales, we obtain that the spatialization sublocale of $L$ is the largest sublocale of $L$ which is spatial. The sublocale $\mathfrak{M}(\pt(L))$ is spatial, since by its definition all its elements are meets of primes. Suppose, now, that $T\se L$ is a spatial sublocale. Then, all its elements are meets of primes, and so we have $\mathfrak{M}(\pt(T))\se T$. We need to have $\pt(T)\se \pt(L)$. Since the operator $\mathfrak{M}$, too, is monotone, we also have $\mathfrak{M}(\pt(T))\se \mathfrak{M}(\pt(L))$. So, $T\se \mathfrak{M}(\pt(L))$.
\end{proof}

From now on, we will refer to $\mathfrak{M}(\pt(L))$ as $\spa (L)$ for any frame $L$. There is an alternative way of seeing the spatialization of a sublocale. 
\begin{proposition}
For every sublocale $S\se L$ we have that $\spa (S)=\bve \{\bl(p):p\in \pt(S)\}$.
\end{proposition}
\begin{proof}
By item (5) of \ref{manylocalethings}, for any sublocale $S\se L$ we have that $\bve \{\bl (p):p\in \pt(L)\}=\mathfrak{M}(\pt(S))$.
\end{proof}

Let us now show that the map $\spa:\SL\ra \SL$ is an interior operator. In the frame theoretical case, in contrast with the d-frame case, we can prove that this interior operator preserves finite joins too. As we will see, this implies that the collection of its fixpoints is actually a cosublocale of $\SL$. Before we prove this, let us see a more general result.

\begin{proposition}\label{sobspaadj8}
For every frame $L$ there is an adjunction of posets as below.
\[
\begin{tikzcd}[row sep=large, column sep = large]
\ca{P}(\pt(L))
\arrow[swap, r, bend left=40, "\mathfrak{M}"{name=F}]
& \SL
\arrow[l, bend left=40, "\pt"{name=F1}]
\arrow[phantom,"\dashv" rotate=-90,from=F1, to=F]
\end{tikzcd}
\]
\end{proposition}
\begin{proof}
It is clear from their definition that both these maps are monotone. To show that we have an adjunction between posets, it is sufficient to show that for any subspace $Y\se \pt(L)$ and any sublocale $S\se L$ we have that $\mathfrak{M}(Y)\se S$ if and only if $Y\se \pt(L)$. Suppose that the former holds. Then, since $\pt$ is monotone we have $\pt(\mathfrak{M}(Y))\se \pt(S)$. We also have that $Y\se \pt(\mathfrak{M}(Y))$ since $\pt(Y)=Y$ by assumption and $\pt(Y)\se \pt(\mathfrak{M}(Y))$ by monotonicity of $\pt$. Then, the first of the desired implications holds. 
Suppose, now, that the latter holds. Since $\mathfrak{M}$ is monotone, we deduce that $\mathfrak{M}(Y)\se \mathfrak{M}(\pt(S))$. Since $\mathfrak{M}(\pt(S))\se S$ as the left hand side is just its spatialization sublocale, we are done. Then, indeed we have an adjunction $\mathfrak{M}\dashv \pt$. 
\end{proof}

\begin{observation}\label{obsejoins}
If we have a frame $L$ and a collection of sublocales $S_1,...,S_n$ of this frame, we have that by item (1) of Lemma \ref{manylocalethings} their join in $\SL$ is computed as $\mathfrak{m}(S_1\cup ...\cup S_n)$. Since primes cannot be nontrivial meets, this means that $\pt(S_1\ve ...\ve S_n)=\pt(S_1)\cup ...\cup \pt(S_n)$.
\end{observation}
\begin{proposition}\label{spaconucleus}
Spatialization $\spa:\SL\ra \SL$ is a conucleus: it is an interior operator which preserves finite joins.
\end{proposition}
\begin{proof}
For any sublocale $S\se L$, we have that $\pt(S)\se S$. Since sublocales are closed under meets, this implies that $\spa (S)\se \mathfrak{M}(S)=S$. Then, spatialization is decreasing in $\SL$. It is also idempotent: if $S$ is a spatial sublocale, all its elements are meets of primes and so $\spa (S)=S$. Now, let us show that the map preserves finite joins of $\SL$. The map $\mathfrak{M}$ preserved all colimits, since it is a left adjoint. On the other hand, the map $\pt$ preserves finite colimits by Observation \ref{obsejoins}. Then, their composition $\mathfrak{M}\circ \pt$ preserves finite joins.
\end{proof}

\begin{corollary}
For every frame $L$ the lattice of its spatial sublocales $\spa [\SL]$ is a coframe. Spatialization is a coframe map $\spa:\SL\ra \spa[\SL]$.
\end{corollary}
\begin{proof}
Both parts of the claim holds because $\spa$ is a conucleus on $\SL$.
\end{proof}

We now claim that the cosublocale $\spa [\SL]\se \SL$ corresponds exactly to the spatialization cosublocale of $\SL$. 

\begin{lemma}\label{booleanareprimes}
The primes of $\SLop$ are exactly the sublocales of the form $\bl(p)$ for some prime $p\in L$.
\end{lemma}
\begin{proof}
Suppose we have a sublocale $S\se L$ with more than one element, say $\{1,x,y\}\se S$ with all three these elements pairwise different. Suppose without losing generality that $x\nleq y$. This means that $y\notin \cl(x)$. We also have $x\ra x=1\neq x$, and so $x\notin \op (x)$. Since $\cl(x)\ve \op (x)=L$, we have that $S\se \op (x)\ve \cl(x)$, but we also have $S\nsubseteq \op (x)$ and $S\nsubseteq\cl (x)$, with witnesses $x$ and $y$, respectively. Then, $S$ is not join prime. Conversely, if $S=\{1,p\}$ and $S\se T\ve U$, in particular $\{p\}\se T\ve U=\mathfrak{m}(T\cup U)$. This implies that $\{p\}\se T\cup U$ since primes cannot be nontrivial meets. We must have either $p\in T$ or $p\in U$. If the former holds then $\{1,p\}\se T$ and if the latter holds then $\{1,p\}\se U$. Then, $\{1,p\}$ is join prime.
\end{proof}

\begin{proposition}\label{spatializationsublocaleassembly}
For every frame $L$ the coframe $\spa[\SL]$ is the spatialization cosublocale of $\SL$.
\end{proposition}
\begin{proof}
The spatialization sublocale of a frame is the closure under meets of the ordered collection of its prime elements. Then the spatialization cosublocale of $\SL$ is the closure under joins of the ordered collection of its join primes, which by Lemma \ref{booleanareprimes} is $\{\bl(p):p\in \pt(L)\}$. 
\end{proof}

As a corollary, we obtain this well-known fact.
\begin{corollary}\label{totspaassemblyspa}
For any frame $L$, the frame $\SLop$ is spatial if and only if all sublocales of $L$ are spatial.
\end{corollary}

\section{The coframe of sobrifications}
Suppose that we have the spectrum $\pt(L)$ of a frame $L$. We want to see sobrification as a closure operator on $\ca{P}(\pt(L))$ in the same way in which we have seen spatialization as an interior operator on $\SL$. In particular, we will see that there is a symmetry between spatialization on $\SL$ and sobrification on $\ca{P}(\pt(L))$. When we have an extremal monomorphism $Y\rightarrowtail X$ in $\bd{Top}$ we can turn it into a concrete subspace inclusion $Y\se X$. Recall that in a category $\ca{C}$ a morphism $f:A\ra B$ is an extremal monomorphism if it is a monomorphism and if $f=ge$ with $e$ an epimorphism, $e$ must be an isomorphism. It is very well-known that the epimorphisms in the category of topological spaces are exactly the surjections. The following result is also quite well-known, but let us prove it. 

\begin{lemma}\label{imagehomeo8}
In the category of topological spaces, a continuous function $f:Y\ra X$ is an extremal monomorphism if and only if it is an injection such that the map $Y\ra f[Y]$ defined as $y\mapsto f(x)$ is a homeomorphism.
\end{lemma}
\begin{proof}
Suppose that $f:Y\ra X$ is a continuous injection between topological spaces such that $f[Y]$ is homeomorphic to $Y$. Suppose that $f=g\circ e$ with $e$ an epimorphism, as depicted below.
\[
\begin{tikzcd}
{} & Z
\arrow{dr}{g} \\
Y 
\arrow{ur}{e} 
\arrow{rr}{f} 
&& X
\end{tikzcd}
\]
If we have $e(y_1)=e(y_2)$ for distinct points $y_1\neq y_2$ of $Y$, we also would have $f(y_1)=f(y_2)$ by the commutativity of the diagram. contradicting that $f$ is an injection. Since $e$ is an epimorphism, it is also a surjection. Then, $e$ must be a bijection. It only remains to show that it is an open map. Let us observe that $g$ cannot identify distinct elements of $e[Y]=Z$, or we would contradict injectivity of $f$. This means that for every subset $S\se Z$ we have that $g^{-1}(g[S])=S$. In particular, $g^{-1}(f[Y])=g^{-1}(g[e[Y]])=g^{-1}(g[Z])=Z$. Now, let $U\se Y$ be an open set. We have that $e[U]=g^{-1}(g[e[U]])$, by the observation we have just made. By assumption on $f$ and by commutativity of the diagram, $g[e[U]]=f[U]$ must be an open set of $f[Y]$: it must be of the form $V\cap f[Y]$ for some open $V\se X$. Then, $e[U]=g^{-1}(g[e[U]])=g^{-1}(f[U])=g^{-1}(V\cap f[Y])$ must be $g^{-1}(V)\cap g^{-1}(f[Y])=Z\cap g^{-1}(V)=g^{-1}(V)$, which is open in $Z$ by continuity of $g$. Then, $e[U]$ is open and this means that $e$ is an open continuous bijection, hence an isomorphism in $\bd{Top}$. 

For the converse, we show that if $f[Y]$ is not isomorphic to $Y$, we can find a factorization $f=g\circ e$ where $e$ is an epimorphism but not an isomorphism. Suppose that $f$ is a continuous injection but also that there is some open $U\se Y$ such that $f[U]$ is not open in $f[Y]$. Now let us define $Z$ to be an isomorphic copy of $f[Y]$. Let $e:Y\ra Z$ be the bijection defined as $y\mapsto f(y)$, and let $g:Z\ra X$ be subspace inclusion. We have that $e$ is a surjection, hence an epimorphism, but it is not an isomorphism since $f[U]$ is open and $U$ is not. By definition of $Z$ the triangle depicted above commutes. 
\end{proof}

It is also known that the subspace inclusions $Y\inclu X$ in the category of topological spaces are essentially the extremal monomorphisms, in the sense that on the one hand all subspace inclusions are extremal monomorphisms, and on the other hand every extremal monomorphism is a subspace inclusion up to homeomorphism, in the following sense: every extremal monomorphisms $f:Y\rightarrowtail X$ in $\bd{Top}$ is such that the map $Y\ra f[Y]$ acting as $y\mapsto f(y)$ is a homeomorphism making the following diagram commute
\[
\begin{tikzcd}
f[Y]
\arrow[dr,"i",hookrightarrow] \\
Y
\arrow{u}{f'} 
\arrow[r,rightarrowtail,"f"] 
& X
\end{tikzcd}
\]
It is clear that any subspace inclusion $Y\se X$ is such that the image of $Y$ under this map is homeomorphic to $Y$, since it is identical to $Y$. We have already seen the fact that the sobrification map $\psi_X:X\ra \pt(\Om(X))$ of a space $X$ is not in general bijective. When $X$ is a $T_0$ space, the map is always injective. In fact, we have more. We know that when $X$ is $T_0$, this map in general is an extremal monomorphism.

\begin{lemma}
If $X$ is a $T_0$ space, the sobrification map $\psi_X:X\rightarrowtail \pt(\Om(X))$ is an extremal monomorphism.
\end{lemma}
\begin{proof}
Suppose that $X$ is a $T_0$ space. Its sobrification map is then a continuous injection. By Lemma \ref{imagehomeo8}, then, to show our claim it suffices to show that for every open  $U\se X$ the set $\psi_X[U]$ is open in $\psi_X[X]\se \pt(\Om(X))$. Suppose, then, that $U\se X$ is an open set. Recall that a typical open of $\pt(\Om(X))$ is of the form $\{f\in \pt(\Om(X)):f(U)=1\}$ for some open $U\se X$, and so an open of the subspace $\psi_X[X]$ is a set of the form $\{N_x\in \pt(\Om(X)):N_x(U)=1\}$ for some open $U\se X$. The set $\psi_X[U]$ is $\{N_x\in \pt(\Om(X)):x\in U\}=\{N_x\in \pt(\Om(X)):N_x(U)=1\}$, and so it is open in $\psi_X[X]$. Then, the sobrification map $\psi_X$ is an extremal monomorphism.
\end{proof}
This means that whenever we have $T_0$ topological space, we may see it up to homeomorphism as a subspace of its sobrification. 
\begin{lemma}\label{forward8}
If $f:X\ra Y$ is an injective function its forward image map preserves intersections.
\end{lemma}
\begin{proof}
Suppose that $f:X\ra Y$ is an injective function. The inclusion $f[A\cap B]\se f[A]\cap f[B]$ follows from the fact that since $f[-]:\ca{P}(X)\ra \ca{P}(Y)$ is monotone we have both $f[A\cap B]\se f[A]$ and $[A\cap B]\se f[B]$. For the other inclusion, suppose that $y\in f[A]\cap f[B]$. Then, there are $a\in A$ and $b\in B$ with $f(a)=f(b)=y$. By injectivity of $f$ we must then have that $a=b\in A\cap B$, and so $y\in f[A\cap B]$. 
\end{proof}

\begin{lemma}\label{compoextre}
The composition of extremal monomorphisms in $\bd{Top}$ is an extremal monomorphism.
\end{lemma}
\begin{proof}
Suppose that $m_1:X\ra Y$ and $m_2:Y\ra Z$ are extremal monomorphisms in $\bd{Top}$. By Lemma \ref{imagehomeo8}, it suffices to show that $m_2\circ m_1$ is a continuous injection such that $m_2[m_1[X]]$ is homeomorphic to $X$. Since continuous functions as well as injections are closed under composition, it suffices to show that $m_2\circ m_1$ is such that forward images of opens of $X$ are opens of $m_2[m_1[X]]$. Suppose that $U\se X$ is open. We have that $m_1[U]$ is open in $Y\cap m_1[X]$, let $V\se Y$ be such that $m_1[X]=V\cap m_1[X]$. The map $m_2$ is an injection and so by Lemma \ref{forward8} we have $m_2[m_1[U]]=m_2[V\cap m_1[X]]=m_2[V]\cap m_2[m_1[X]]$. By assumption of $m_2$, we have that $m_2[V]$ must be $W\cap m_2[Y]$ for some open $W\se Z$. This means that $m_1[X]=W\cap m_2[Y]\cap m_2[m_1[Y]]$. Since $m_2$ is an injection this is also the same as $W\cap m_2[Y\cap m_1[X]]=W\cap m_2[m_1[X]]$. Then, indeed $m_2[m_1[U]]$ is an open set of the subspace $m_2[m_1[X]]\se Z$. So, the map $m_2\circ m_1$ is an extremal monomorphism.  
\end{proof}

\begin{lemma}\label{sobrificationmono2}
Suppose that $X$ is a sober space, and that $\psi_X^{-1}:\pt(\Om(X))\ra X$ is the inverse of the sobrification map. For any subspace $i_Y:Y\se X$ the map $\psi_X^{-1}\circ \pt(\Om(i_Y)):\pt(\Om(Y))\rightarrowtail X$ is an extremal monomorphism.  
\end{lemma}
\begin{proof}
By Lemma \ref{compoextre}, it suffices to show that isomorphisms are extremal monomorphisms. This follows from Lemma \ref{imagehomeo8}, since for a continuous open bijection $f:X\cong Y$ indeed we have that the image $f[X]$ is homeomorphic to $X$.
\end{proof}

In the lemma below we have kept the convention that for a continuous function $f:Y\rightarrow X$, we have denoted as $f'$ the map $Y\ra f[Y]$ acting as $y\mapsto f(y)$.

\begin{lemma}\label{thereissoby}
Suppose that $X$ is a sober space, and that $Y\se X$ is a subspace. We have a chain of subspace inclusions $Y\se \sob(Y)\se X$ such that the following diagram in $\bd{Top}$ commutes. All vertical arrows are homeomorphisms.
\[
\begin{tikzcd}
Y
\arrow[rr,hookrightarrow]
&& \sob(Y)
\arrow[rr,hookrightarrow]
&& X
\\
Y
\arrow[u,"\id_Y"]
\arrow[rr,rightarrowtail,"\psi_Y"]
&& \pt(\Om(Y))
\arrow[u,"(\psi_X^{-1}\circ \pt(\Om(i)))'"]
\arrow[rr,rightarrowtail,"\pt(\Om(i))"]
&& \pt(\Om(X))
\arrow[u,"\psi_X^{-1}"]
\end{tikzcd}
\]
\end{lemma}
\begin{proof}

The map $\psi_X^{-1}\circ \pt(\Om(i))$ is an extremal monomorphism, by Lemma \ref{sobrificationmono2}. By the observation made after Lemma \ref{imagehomeo8}, then, there is a subspace $\sob(Y)$ of $X$ such that the subspace inclusion $\sob(Y)\inclu X$ makes the right square of the rectangle commute. Let us show that the left square commutes. We need to show that for every $y\in Y$ we have that $\psi_X^{-1}(\pt(\Om(i))(\psi_Y(y)))=y$ for every $y\in Y$. We notice that for the subspace inclusion $i:Y\inclu X$ we have that $\Om(i):\Om(X)\ra \Om(Y)$ is the map $i^{-1}:U\mapsto U\cap Y$. Then, the map $\pt(\Om(i)):\pt(\Om(Y))\ra \pt(\Om(X))$ is precomposition with $-\cap Y$, that is, it assigns to each frame map $f:\Om(X)\ra 2$ the map $\Om(Y)\ra 2$ defined as $U\mapsto f(U\cap Y)$. In particular, for a map of the form $N_y:\Om(Y)\ra 2$, we will have that $\pt(\Om(i))(N_y)$ is a map $\Om(X)\ra 2$ such that it gives $1$ whenever $y\in U\cap Y$, and it gives $0$ whenever $y\notin U\cap Y$. It then coincides with the neighborhood map $N_y:\Om(X)\ra 2$. Then, we have the following.
\begin{align*}
    & \psi_X^{-1}(\pt(\Om(i))(\psi_Y(y)))=\\
    & =\psi_X^{-1}(\pt(\Om(i))(N_y))=\\
    & =\psi_X^{-1}(N_y)=\\
    & =y.\qedhere
\end{align*}
\end{proof}

 In light of this lemma, when we have a sober space $X$ and a subspace $Y$, we will consider the sobrification of $Y$ to be a subspace $\sob(Y)$ of $X$ with $Y\se \sob(Y)$ the sobrification map. For any sober space $X$ and any subspace $Y\se X$, we may calculate its sobrification $\sob(Y)$ as a closure operator on $\ca{P}(X)$. Let us remind the reader that we have already shown that the sober subspaces of a sober space form a closure system: for a sober space $X$, the sober subspaces coincide with those whose underlying set is a Skula-closed set.
 
\begin{lemma}\label{intersectionofsobers}
If $X$ is a sober space and $Y\se X$ a subspace, we have
\[
\sob(Y)=\bca \{Z\se X:Y\se Z,Z\mb{ is sober}\}
\]
of $X$ with sobrification map the set inclusion $Y\se \sob(Y)$.
\end{lemma}
\begin{proof}
By Lemma \ref{thereissoby}, we know that there is a subspace $\sob(Y)\se X$ such that it is isomorphic to the sobrification of $Y$, with the sobrification map given by the set inclusion $i_{YS}:Y\inclu \sob(Y)$. By the universal property of the sobrification of a space, whenever $Z$ is a sober subspace of $X$ such that we have an inclusion $i_{YZ}:Y\se Z$, there must be a unique map $i_{SZ}:\sob(Y)\ra Z$ such that the following commutes.
\[
\begin{tikzcd}
\sob(Y)
\arrow[dr,"i_{SZ}"] \\
Y
\arrow[u,hookrightarrow,"i_{YS}"] 
\arrow[r,hookrightarrow,"i_{YZ}"] 
& Z
\end{tikzcd}
\]
Because this triangle commutes, $i_{SZ}$, too, must be a subset inclusion. Then, every sober subspace of $X$ which contains $Y$ must be a superset of $\sob(Y)$, and so $\sob(Y)\se \bca \{Z\se X:Z\mb{ is sober, }Y\se Z\}$. For the other set inclusion, it suffices to notice that $\sob(Y)\in \{Z\se X:Z\mb{ is sober, }Y\se Z\}$.
\end{proof}

We now look at the specific case where we want to sobrify subspaces of a sober space of the form $\pt(L)$ for some frame $L$.

\begin{lemma}\label{sobclo}
For any frame $L$ and any subspace $Y\se \pt(L)$ we have that its sobrification is $\pt(\mathfrak{M}(Y))$, with sobrification map the set inclusion $Y\se \pt(\mathfrak{M}(Y))$
\end{lemma}
\begin{proof}
For this, it suffices to show that for every subspace $Y\se \pt(L)$ the intersection $\sob(Y)$ of all sober subspaces of $\pt(L)$ which contain it is $\pt(\mathfrak{M}(Y))$. One set inclusion is clear, since $\pt(\mathfrak{M}(Y))$ is sober and $Y\se \pt(\mathfrak{M}(Y))$. For the other set inclusion, suppose that $Y\se Z$ with $Z\se \pt(L)$ sober. We may then write $Z=\pt(S)$ for some sublocale $S\se L$. We then have that $\pt(\mathfrak{M}(Y))\se \pt(\mathfrak{M}(\pt(S)))$ by monotonicity of both $\pt$ and $\mathfrak{M}$, and since the points of a frame are the same as the points of its spatialization we have $\pt(S)=\pt(\mathfrak{M}(\pt(S)))$ and so $\pt(\mathfrak{M}(Y))\se \pt(S)$.
\end{proof}
Finally, we may come back to the adjunction introduced in Proposition \ref{sobspaadj8}. The next proposition will make precise our opening statement for this section; that there is a symmetry between spatialization of sublocales and sobrification of subspaces.

\begin{proposition}\label{sobspaadjuseful}
For every frame $L$, consider the following adjunction.
\[
\begin{tikzcd}[row sep=large, column sep = large]
\ca{P}(\pt(L))
\arrow[swap, r, bend left=40, "\mathfrak{M}"{name=F}]
& \SL
\arrow[l, bend left=40, "\pt"{name=F1}]
\arrow[phantom,"\dashv" rotate=-90,from=F1, to=F]
\end{tikzcd}
\]
The fixpoints of $\mathfrak{M}\circ \pt$ are exactly the spatial sublocales, and the fixpoints of $\pt\circ \mathfrak{M}$ are exactly the sober subspaces.
\end{proposition}
\begin{proof}
This follows from Proposition \ref{sobspaadj8} combined with Lemmas \ref{spaint} and \ref{sobclo}.
\end{proof}

\section{The relation between subspaces and sublocales}

 We can now finally see that we can exploit the equivalence between sober spaces and spatial frames to show that the coframes of sober subspaces of $\pt(L)$ and that of spatial sublocales of $L$ are the same.
\begin{proposition}\label{spaisosober}
For any frame $L$ we have that the map $\pt:\spa [\SL]\ra\sob[\ca{P}(\pt(L))]$ is a coframe isomorphism. Its inverse acts as $S\mapsto \mathfrak{M}'(S)$.
\end{proposition}
\begin{proof}
Proposition \ref{sobspaadj8} tells us that we have an adjunction $\mathfrak{M}:\ca{P}(\pt(L))\lra \SL:\pt$ with $\mathfrak{M}\dashv \pt$. This adjunction restricts to an equivalence of categories between the two full subcategories determined by the objects which are fixpoints of $\mathfrak{M}\circ \pt$ and $\pt\circ \mathfrak{M}$. By Proposition \ref{sobspaadjuseful}, these are the spatial sublocales and the sober subspaces, respectively.
\end{proof}

We can gather together this result and the results of the two previous subsections in the following picture. The following is a commutative diagram in the category of coframes. We will refer to this as the ``main diagram" of this section.
\begin{center}
\begin{tikzcd}[row sep=large, column sep = large]
\spa [\mf{S}(L)]  
\arrow{r}{\pt(\cong)}  
& \sob[\ca{P}(\pt(L))] 
\arrow[hookrightarrow]{d}  \\
\mf{S}(L) 
\arrow{r}{\pt} 
\arrow[twoheadrightarrow]{u}{\spa} 
& \ca{P}(\pt(L))  
\end{tikzcd}
\end{center}

Our analysis of the connection between subspaces and sublocales will be centred around analyzing what happens when the vertical arrows in this diagrams are isomorphisms.

\begin{itemize}
    \item When the arrow on the left is an isomorphism, we will have that the collection of sublocales perfectly represent the sober subspaces of $\pt(L)$. This condition is equivalent to all sublocales of $L$ being spatial, and this in turn to $\SLop$ itself being spatial. This question of how to characterize the frames with this property, usually referred to as the \tc{totally spatial} ones, has been explored in several papers. Most famously, Niefield and Rosenthal in \cite{niefield87} have proven that this is equivalent to a certain condition on the primes of $L$. More recently (see \cite{Avila19} and \cite{avila2019frame}), the question has been explored using a new approach involving Esakia duality. In \cite{ISBELL91} the question is explored, among other things, and a more point-set approach is used. We will re-prove some of their results, with the modification that our versions of Niefield and Rosenthal's results will not rely on Zorn's lemma, and that our proof techniques will use the approach of sublocales as concrete subsets of frames.
    \item When the arrow on the right is an isomorphism, the spatial sublocales of $\SL$ perfectly represent the subspaces of $\pt(L)$. We will show that this is equivalent to the space $\pt(L)$ being $T_D$ and the primes of $L$ being \tc{weakly covered}. This is a weak version of the notion of coveredness which was introduced in \cite{banaschewski15}, and then used in \cite{picado19} in relation to $T_D$ spaces. In spatial frames coveredness is equivalent to weak coveredness.
    \item When both the arrow on the left and that on the right are isomorphisms, the sublocales perfectly represent the subspaces. The frames $L$ satisfying this condition are always spatial. In \cite{simmons80} it is shown that $\Om(X)$ is such that $\SO\cong \ca{P}(X)$ if and only if the space $X$ is \tc{scattered}. We will prove this result in more algebraic terms involving the primes of $L$, and show that scatteredness of $X$ amounts to a natural strengthening of the condition of \cite{niefield87} on the primes. 
\end{itemize}
 
\subsection{When are all subspaces sober?}

For a frame $L$, we say that a prime $p\in \pt(L)$ is \tc{weakly covered} if whenever $p=\bwe P$ for $P\se \pt(L)$ we have $p\in P$. This is a weakening of the notion of \tc{covered prime}, by which is meant a prime $p\in L$ such that whenever $p=\bwe A$ we have $p\in A$. The notion of covered prime is already present in relation with the $T_D$ axiom in the literature, for instance in \cite{picado19} it is observed that $X$ being a $T_D$ space is equivalent to all primes of $\Om(X)$ of the form $p_x$ for some $x\in X$ being covered. Our notion of weakly covered stems from the fact that we want a notion of covered prime generalized to nonspatial frames.
\begin{lemma}\label{allsoberimpli1}
For a frame $L$, every prime of $L$ is weakly covered if and only if all subspaces of $\pt(L)$ are sober.
\end{lemma}
\begin{proof}
Suppose that all primes of $L$ are weakly covered. For every subspace $Y\se \pt(L)$ we have that $\mathfrak{M}(Y)$ contains the same primes as $Y$, since we cannot get any new prime as an meet $\bwe P$ of primes $P\se Y$, or this would be a prime which is not weakly covered. Then, $Y=\pt(\mathfrak{M}(Y))$ for every subspace $Y\se \pt(L)$. By Proposition \ref{sobclo} this means that every subspace of $\pt(L)$ is sober. Conversely, suppose that there is a prime $p\in \pt(L)$ which is not weakly covered. Let $P\se \pt(L)$ be a collection of primes with $p\notin P$ and such that $\bwe P=p$. We have that $p\in \pt(\mathfrak{M}(P))$ but $p\notin P$, and so the subspace $P\se\pt(L)$ is not its own sobrification.
\end{proof}
The condition that all subspaces of $\pt(L)$ are sober is equivalent to the right vertical arrow of the main diagram of this section being an isomorphism. This is also equivalent to having that $\pt$ establishes an isomorphism of coframes $\spa[\SL]\cong \ca{P}(\pt(L))$. Let us look at some more equivalent ways of phrasing this condition.

\begin{lemma}\label{coveredisolated}
For a frame $L$, prime $p\in \pt(L)$ is weakly covered if and only if the point $\bl(p)$ is isolated in $\pt(\SLop)$.
\end{lemma}
\begin{proof}
Suppose that $L$ is a frame and that $p\in \pt(L)$ be a weakly covered prime. Let us consider the singleton $\{\bl(p)\}\se \pt(\SLop)$. A typical open of $\pt(\SLop)$ is a set of the form $\{\bl(q): \bl(q)\nsubseteq T\}$, with $q$ ranging over the primes of $L$. Let $\{q_i:i\in I\}$ be the set of primes $\pt(L){\sm}\{p\}$.  We claim that $\{\bl(p)\}=\{\bl(q):\bl(q)\nsubseteq \bve_i \bl(q_i)\}$. Since $\bve_i \bl(q_i)=\mathfrak{M}(\{q_i:i\in I\})$, by weak coveredness of $p$ we have $p\notin \bve_i \bl(q_i)$, that is, $\bl(p)\subseteq \bve_i \bl(q_i)$. For all other primes $q\in L$, by definition of the collection $q_i$ we must have $\bl(q)\se \bve_i \bl(q_i)$. Then, indeed $\{\bl(p)\}$ is isolated. For the converse, let us suppose that $\{\bl(p)\}$ is isolated. Let $S$ be the sublocale such that $\bl(p)\nsubseteq S$ and such that $\bl(q)\se S$ whenever $q$ is a prime other than $p$. Suppose that $Q$ is a collection of primes not containing $p$. We claim that $\bwe Q\neq p$. By property of the sublocale $S$, we must have $\bve \{\bl(q):q\in Q\}\se S$, and by meet closure of sublocales this implies $\bwe Q\in S$. Then, we must have $\bwe Q\neq p$, or we would have $p\in S$, contradicting our hypothesis.
\end{proof}

\begin{lemma}
For a frame $L$, a prime $p\in \pt(L)$ is weakly covered if and only if it is an isolated point of the subspace $\pt(\up p)=\up p\cap \pt(L)$ of $\pt(L)$.
\end{lemma}
\begin{proof}
Suppose that $p\in L$ is a weakly covered prime. Consider the space $\up p\cap \pt(L)\se \pt(L)$. Let $P$ be the collection $\{q\in \up p\cap \pt(L):q\neq p\}$. By assumption that $p$ is weakly covered, we have that $p<\bwe P$, and so $\bwe P\nleq p$. This means that the open $\{q\in \pt(\up p):\bwe P\nleq q\}\se \up p\cap \pt(L)$ only contains the point $p$. Conversely, suppose that $p\in \pt(L)$ is not weakly covered, and let $P$ be a collection of primes with $\bwe P=p$. Consider the subspace $\pt(\up p)\se \pt(L)$. Suppose that there is an open $\{q\in \pt(\up p): a\nleq q\}$ containing $p$. We show that it must contain at least one other point of $\pt(\up p)$. If we had $a\leq q$ for all primes $q\neq p$ with $p\leq q$, we would also have that $a\leq \bwe \{q\in \pt(\up p):q\neq p\}\leq \bwe P=p$, contradicting that $a\nleq p$. Then, the point $p$ is not isolated in $\pt(\up p)$.  
\end{proof}

Let us now recall the fact that $\pt(\SLop)$ is the Skula space of $\pt(L)$. We also remind the reader of the fact that for a frame $L$ the Skula-closed sets of its spectrum coincide with the underlying sets of its sober subspaces. We will now use these facts to explore when $\pt(L)$ is $T_D$ in terms of the topology on $\pt(\SL^{op})$.

\begin{lemma}\label{atomsofba}
In a complete Boolean algebra $B$ the meet prime elements coincide with the coatoms.
\end{lemma}
\begin{proof}
Suppose that $B$ is a complete Boolean algebra. Suppose that $p\in \pt(B)$, and suppose towards contradiction that there is $x\in B$ with $p<x<1$. Since $x\neq 1$ and since $B$ is a Boolean algebra we must have $\neg x\neq 0$. We have that $\neg x\leq \neg p$. We cannot have $\neg x \leq \neg p$, otherwise we would have $\neg x\leq p\we \neg p$, contradicting that $\neg x\neq 0$. We deduce $\neg x\nleq \neg p$. By assumption, we also have $x\nleq p$, and by primality of $p$ this means that $x\we \neg x\nleq p$, which is a contradiction. Conversely, suppose that $p\in B$ is a coatom. Suppose that $x\we y\leq p$. Using distributivity we have $(x\ve p)\we (y\ve p)=p$. Since $p$ is a coatom, we have $x\ve p=p$ or $x\ve p=1$, and similarly for $y\ve p$. We cannot have $x\ve p=1=y\ve p$ or we would contradict our assumption that $(x\ve p)\we (y\ve p)=p$. Then, we must have either $x\ve p=p$, that is $x\leq p$, or $y\ve p=p$, that is $y\leq p$.
\end{proof}

\begin{lemma}\label{spatialba}
The prime spectrum of a Boolean algebra is the discrete space whose points are its coatoms.
\end{lemma}
\begin{proof}
Suppose that $B$ is a complete Boolean algebra which is also a spatial frame. Its prime spectrum $\pt(B)$ is the set of its coatoms, by Lemma \ref{atomsofba}. We claim that every coatom is an isolated point of $\pt(B)$. Suppose, then, that $p\in \pt(B)$. Since $B$ is a Boolean algebra and $p\neq 1$ we must also have $\neg p\neq 0$. This means that $\neq p\nleq \neg \neg p=p$. If $q\in \pt(L)$ is such that $q\neq p$, we have that $p\nleq q$, since $p$ and $q$ are distinct coatoms, and that $p\we \neg p\leq q$. Since $q$ is a prime, this implies that $ p\leq q$. Then, we have that the open $\{q\in \pt(B):\neg p\nleq q\}\se \pt(B)$ consists only of the singleton $\{p\}$. Then, $p$ is isolated.
\end{proof}

\begin{lemma}\label{soberclopenopen}
If a sober space has a frame of opens which is a Boolean algebra, it must be discrete.
\end{lemma}
\begin{proof}
If a sober space is such that its frame of opens is Boolean, it must be homeomorphic to the prime spectrum of a Boolean algebra, which we know to be discrete, by Lemma \ref{spatialba}.
\end{proof}

\begin{lemma}\label{lastimpli1}
For a frame $L$, we have that $\spa[\SL]$ is Boolean if and only if the space $\pt(\SLop)$ is discrete.
\end{lemma}
\begin{proof}
If $\pt(\SLop)$ is a discrete space, its frame of opens is a Boolean algebra. By Proposition \ref{spatializationsublocaleassembly}, this frame of opens is anti-isomorphic to the coframe $\spa [\SL]$. For the converse, suppose that $\spa[\SL]$ is a Boolean algebra. Since this is anti-isomorphic to $\Om(\pt(\SLop))$, this means that the frame of opens of $\pt(\SLop)$ is Boolean. By Lemma \ref{soberclopenopen}, the space $\pt(\SLop)$ is discrete. 
\end{proof}

\begin{theorem}\label{important1}
For a frame $L$, the following are equivalent.
\begin{enumerate}
\item All the primes in $L$ are weakly covered.
\item All subspaces of $\pt(L)$ are sober.
\item The map $\pt:\spa [\mf{S}(\Om(X))]\ra \ca{P}(\pt(L))$ is an isomorphism.
\item The space $\pt(\SLop)$ is discrete. 
\item The space $\pt(L)$ is $T_D$.
\item The coframe $\spa [\SL]$ is a Boolean algebra.
\end{enumerate}
\end{theorem}
\begin{proof}
The equivalence between (1) and (2) is Lemma \ref{allsoberimpli1}. As we have observed already, all subspaces of $\pt(L)$ being sober is a condition equivalent to the right vertical arrow of the main diagram of this section being an isomorphism. Then, (2) and (3) are equivalent. By Lemma \ref{coveredisolated}, a prime $p\in L$ being covered is equivalent to the point $\bl(p)$ of the space $\pt(\SLop)$ being isolated, and so all primes of $L$ being weakly covered is equivalent to $\pt(\SLop)$ being discrete. Then, (1) and (4) are equivalent. Since $\pt(\SLop)$ is the Skula space of $\pt(L)$, the space $\pt(\SLop)$ being discrete is equivalent to $\pt(L)$ being a $T_D$ space, and so (4) and (5) are equivalent. Lemma \ref{lastimpli1} states the equivalence of (4) and (6).
\end{proof}

\subsection{When are all sublocales spatial?}

A frame $L$ is said to be \textit{totally spatial} if all its sublocales are spatial. In this subsection we explore different characterizations of this condition, with special attention to what this condition means for the prime elements of $L$. For $M\se L$ a subset of a frame $L$, we say that $m\in M$ is an \textit{essential} element of the meet $M$ if $\bwe M\neq \bwe (M{\sm}\{m\})$, in other words, if $\bwe (M{\sm}\{m\})\nleq m$. In \cite{niefield87}, the authors define a notion of essential prime of a certain element $a$, and the definition is phrased in terms of minimal elements of the set $\pt(\up a)$. The authors observe that in this collection every descending chain has a lower bound, and they deduce that there is a set $\mf{Min}(a)$ of primes above $a$ which are minimal in the collection $\pt(\up a)$. Then they observe that if $a$ is the meet of the primes above it then $a=\bwe \mf{Min}(a)$. Finally, for a frame $L$ and an element $a\in L$ which is the meet of the primes above it, they define $p$ to be an essential prime of some $a\in L$ if $p\in \mf{Min}(a)$ and $\bwe \mf{Min}(a){\sm}\{p\}\neq a$. We decided to adapt their definition and results so that we do not need Zorn's Lemma to phrase them or prove them, since this is not much more complicated than their approach. For a frame $L$ and an element $a\in L$ which is the meet of the primes above it, we say that $p\in \pt(L)$ is an \textit{essential} prime of $a$ if $p$ is an essential element of the particular meet $\bwe \{q\in \pt(\up a):\mbox{ $q=p$ or $p\nleq q$}\}$. Then, $p$ is an essential prime of $a$ if and only if $a\leq p$ and $\bwe (\pt(\up a){\sm}\up p)\nleq p$. Now, let us prove our versions of the results in \cite{niefield87}.

\begin{lemma}\label{essential1chara}
If $L$ is a frame and $a\in L$ is the meet of the primes above it, we have that for every $x\in L$ 
\[
x\ra a=\bwe \pt(\up a){\sm}\up x,
\]
and so a prime $p$ with $a\leq p$ is essential if and only if $p_a\ra a\nleq p_a$.
\end{lemma}
\begin{proof}
Since the element $a\in L$ is assumed to be the meet of the primes above it we have $x\ra a=x\ra \bwe \pt(\up a)=\bwe \{x\ra p:p\in \pt(\up a)\}$. Each of the conjuncts of this meet is of the form $x\ra p$, and so by the property of prime elements stated in item (2) of Lemma \ref{manylocalethings}, we have that if $x\leq p$ it is $1$, and if $x\nleq p$ it is $p$. The meet itself, then is equal to $\bwe \{q\in \pt(\up a): x\nleq q\}$.
\end{proof}

\begin{lemma}\label{essential2chara}
If $L$ is a frame and $a\in L$ is the meet of the primes above it, we have that $p_a$ is an essential prime of $a$ if and only if $p_a=(p_a\ra a)\ra a$.
\end{lemma}
\begin{proof}
We claim that for every prime $p$ with $a\leq p$ we have that $p\ra a\nleq p$ if and only if $p=(p\ra a)\ra a$. Suppose that the former holds. We already know that $p\leq (p\ra a)\ra a$, by item (2) of Corollary \ref{implijoinmeet}. To show the reverse inequality, we show that $z\leq (p\ra a)\ra a$ implies $z\leq p$ for every $z\in L$. If $z\leq (p\ra a)\ra a$, we have that $z\we (p\ra a)\leq a$. Since $a$ is the meet of the primes above it, in particular this means that $z\we (p\ra a)\leq p$. By primality of $p$, this means that either $z\leq p$ or $p\ra a\leq p$. By the second option does not hold by assumption, and so $z\leq p$. Indeed, then, $p=(p\ra a)\ra a$. For the other direction of the main claim, suppose that $(p\ra a)\ra a=p$. Towards contradiction, suppose that $p\ra a\leq p$. By assumption, this implies that $p\ra a\leq (p\ra a)\ra a$. So, $p\ra a\leq a$. Since $-\ra a$ is antitone (by Lemma \ref{impliadj}) this means $1=a\ra a\leq (p\ra a)\ra a=p$. This contradicts $p$ being prime. Then, we must have $p\ra a\nleq p$. We have now proven that for every prime $p\in \pt(\up a)$ we have that $(p\ra a)\ra a=p$ is equivalent to $p\ra a\nleq p$. The claim follows from this fact and from Lemma \ref{essential1chara}.  
\end{proof}

\begin{lemma}\label{essentialiffinboolean}
Suppose that $L$ is a frame and that $a\in L$ is an element such that it is the meet of the primes above it. We have that $p_a$ is an essential prime of $a$ if and only if $p_a\in \bl(a)$. Thus, $p$ is an essential prime of $a$ if and only if it appears in any sublocale in which $a$ appears.
\end{lemma}
\begin{proof}
To see this, we combine Lemma \ref{essential2chara} with item (2) of Lemma \ref{booleanfacts}.
\end{proof}
Let us now explore total spatiality in terms of Boolean sublocales.
\begin{lemma}\label{opandb}
For a frame $L$ and for $x,y\in L$ we have that $\bl(x\ra y)=\op (x)\cap \bl(y)$.
\end{lemma}
\begin{proof}
Let us show both set inclusions. We have that $x\ra y\in \bl(y)$, since by item (1) of Lemma \ref{booleanfacts} the sublocale $\bl(y)$ is $\{z\ra y:z\in L\}$. So, $\bl(x\ra y)\se \bl(y)$. We also have that $x\ra y\in \op (x)$, since we also have $\op(x)=\{x\ra z:z\in L\}$. For the other set inclusion, suppose that we have some $z\in \op (x)\cap \bl(y)$. This tells us two facts: that $x\ra z=z$ and that $z=z'\ra y$ for some $z'\in L$. Combining the two facts we obtain that $z=x\ra (z'\ra y)=(x\wedge z')\ra y=z'\ra (x\ra y)$ and so $z$ must indeed be an element of $\bl (x\ra y)$, by item (1) of Lemma \ref{booleanfacts}.
\end{proof}
\begin{proposition}\label{enoughonly1point}
A frame $L$ is totally spatial if and only if every Boolean sublocale other than $\bl(1)=\{1\}$ has at least a point.
\end{proposition}
\begin{proof}
Suppose that $L$ is a frame such that all its sublocales are spatial. If $\bl(x)$ is a Boolean sublocale with $x\neq 1$, we must have $\mathfrak{M}(\pt(\bl(x)))=\bl(x)$, and in particular $\mathfrak{M}(\pt(\bl(x)))\neq \{1\}$. Then, there must be at least a prime in $\bl(x)$. For the converse, suppose that every nontrivial Boolean sublocale contains at least a prime. Towards contradiction, let us suppose that there is a nonspatial sublocale $S\se L$. Then, we must have $S\nsubseteq \mathfrak{M}(\pt(S))$. In particular, there must be some $s\in S$ such that $s< \bwe_i p_i$, where $\{p_i:i\in I\}=\pt(\up s)\cap S$. Then, we have that $\bwe _i p_i\ra s\neq 1$, and so the Boolean sublocale $\bl(\bwe_i p_i\ra s)$ must be nontrivial, and by our assumption this means that it must contain at least a point. Let this be $p$. By Lemma \ref{opandb}, we have $\bl(\bwe_i p_i\ra s)=\op (\bwe _ip_i)\cap \bl(s)$. Then, we must have both $p\in \op (\bwe_i p_i)$ and $p\in \bl(a)$. The first fact implies that $\bwe _ip_i\ra p=p$, and this implies $\bwe_i p_i\nleq p$ by primality of $p$. The second fact implies that $p=y\ra s$ for some $y\in L$, and this implies that $p\in S$. By definition of the collection $\{p_i:i\in I\}$, our two deductions that $\bwe_ip_i\nleq p$ and that $p$ is a prime with $p\in S$ contradict each other. Then, there cannot be any nonspatial sublocale in $L$.
\end{proof}

\begin{corollary}\label{essentialimportant}
A frame $L$ is totally spatial if and only if it is spatial and all its elements other than $1$ have an essential prime.
\end{corollary}
\begin{proof}
If a frame $L$ is spatial and all its elements have an essential prime, by Lemma \ref{essentialiffinboolean} every element $a\in L$ which is not $1$ is such that there is a prime in $\bl(a)$. Then, the frame is totally spatial by Lemma \ref{enoughonly1point}. Conversely, the frame $L$ is totally spatial, in particular it is spatial. If $a\in L$ is an element other than $1$, we have that $\bl(a)$ contains at least a prime, by Lemma \ref{enoughonly1point}. Since $a$ is the meet of the primes above it, Lemma \ref{essentialiffinboolean} applies and $a$ has an essential prime.
\end{proof}

For a frame $L$ and an element $a\in L$ such that it is the meet of the primes above it, we denote as $\mf{Ess}(a)$ the collection of its essential primes. 

\begin{lemma}\label{meetofess}
A frame is totally spatial if and only if every element $a\in L$ is the meet of its essential primes.
\end{lemma}
\begin{proof}
Let $L$ be a totally spatial frame. As it is the case in every frame, the element $1$ is always the meet of primes above it, and it is also the meet of its essential primes, since it is the empty meet. Let us now consider an element $a\in L$ with $a\neq 1$. The sublocale $\bl(a)$ is spatial, and so it must be the case that every element of $\bl(a)$ is a meet of primes. Since $a$ is the bottom element of $\bl(a)$, in particular, $a=\bwe \pt(\bl(a))$. By Lemma \ref{essentialiffinboolean}, we have $\pt(\bl(a))=\mf{Ess}(a)$. Conversely, if every element in $L$ is the meet of the essential primes above it, in particular the frame $L$ is spatial. For $a\in L$ with $a\neq 1$, the set $\mf{Ess}(a)$ must be nonempty or we would have $\bwe \mf{Ess}(a)=1\neq a$, contradicting our hypothesis.
\end{proof}
The following results are known, but they were proven not in \cite{niefield87} but in \cite{ISBELL91}. There, it is proven that a $T_0$ space is sober and such that $\Om(X)$ has a spatial assembly and that this holds if and only if every closed subspace of $X$ is a discrete dense subspace. We present a proof of this result in terms of prime elements to tie these results more directly with those of \cite{niefield87}.

\begin{lemma}\label{essentialimplidds}
If a frame $L$ is totally spatial every closed subspace of $\pt(L)$ has a discrete dense subspace. In particular, for a closed subspace $\pt(\up a)$, its discrete dense subspace is $\mf{Ess}(a)=\pt(\bl(a))$.
\end{lemma}
\begin{proof}
Suppose that $L$ is totally spatial. Consider an arbitrary closed subspace $\pt(\up a)\se \pt(L)$. We claim that $\pt(\bl(a))=\mf{Ess}(a)$ is a discrete dense subspace. By item (3) of Lemma \ref{booleanfacts}, this subspace is the spectrum of a Boolean algebra, and so it is discrete by Lemma \ref{spatialba}. For density, we show that if there is an open of the space $\pt(\up a)$ such that it is disjoint from $\mf{Ess}(a)$, it must be empty. Suppose, then, that we have a set $\{q\in \pt(a):x\nleq q\}$ which is disjoint from $\mf{Ess}(a)$, that is, such that for every essential prime $p$ of $a$ we have $x\leq q$. Then, $x\leq \bwe \mf{Ess}(a)$. Since $L$ is totally spatial, by Lemma \ref{meetofess} we have $\bwe \mf{Ess}(a)=a$. Then, for every prime $q\in \pt(\up a)$, we have that $x\leq a\leq q$. So, the open set $\{q\in \pt(\up a):x\nleq q\}$ of $\pt(\up a)$ is empty. Then, the space $\mf{Ess}(a)$ must be a dense subspace of $\pt(\up a)$. 
\end{proof}
We would now like to prove the converse of this result.

\begin{lemma}\label{firstprime}
For a frame $L$ and a subspace $P\se \pt(L)$ we have that $p\in P$ is isolated in $P$ if and only if $p$ is an essential element of the meet $\bwe P$. 
\end{lemma}
\begin{proof}
Suppose that $p\in P$ is an isolated point. Opens in $\pt(L)$ are of the form $\{p\in \pt(L):a\nleq p\}$ for some element $a\in L$. The point $p\in P$ being isolated, then, means that there is some $a\in L$ such that $a\nleq p$, and such that $a\leq q$ whenever $q\in P$ and $q\neq p$. This means that $\bwe P{\sm}\{p\}\nleq p$, or we would have that $a\leq \bwe P{\sm}\{p\}\leq p$, which would contradict $a\nleq p$. Conversely, if $p$ is essential in the meet $\bwe P$, we have that the open $\{q\in \pt(L):\bwe P{\sm}\{p\}\nleq q\}$ contains $p$ but is disjoint from $P{\sm}\{p\}$. Then, $p\in P$ is an isolated point. 
\end{proof}

\begin{lemma}\label{not<}
For a frame $L$ and a family $P\se \pt(L)$ of primes in which every prime is essential for the meet $P$, if $p,q\in P$ are distinct they are incomparable. 
\end{lemma}
\begin{proof}
Suppose that $L$ is a frame and that $P\se \pt(L)$ is such that in the meet $\bwe P$ all elements are essential. Suppose that we have $p,q\in P$ two distinct primes. Suppose that $p\leq q$. We deduce that $\bwe P{\sm}\{q\}=\bwe P{\sm}\{q\}\we p\leq q$. 
\end{proof}

\begin{proposition}\label{ddsimpliests}
For a frame $L$, if $L$ is spatial and every closed subspace of $\pt(L)$ has a discrete dense subspace, the frame is totally spatial.
\end{proposition}
\begin{proof}
It suffices to show that if $L$ is spatial and every closed subspace of $\pt(L)$ has a discrete dense subspace, every element $a\in L$ has an essential prime. Let us assume the antecedent. Suppose that we have an element $a\in L$. The closed subspace $\pt(\up a)$ has a discrete dense subspace, say this is $P$. First, we claim that $\bwe P\leq a$. By spatiality, it suffices to show that $q\leq \bwe P$ implies $q\leq a$ for every prime $q\in \pt(L)$. So, suppose that the antecedent holds. As $P\se \pt(\up a)$ is dense, and as the open $\{r\in \pt(\up a):q\nleq r\}$ is disjoint from $P$, we must have that $q\leq \bwe \pt(\up a)=a$. Indeed, then $\bwe A\leq a$. The space $P$ must be nonempty, so let $p\in P$. Since $P$ is discrete, $p\in P$ is an isolated point of this space, and so we have $\bwe P{\sm}\{p\}\nleq p$, by Lemma \ref{firstprime}. We show that $\bwe \pt(\up a){\sm}\up a$ by showing that $\bwe P{\sm}\{p\}=\bwe \pt(\up a){\sm}\up p$. By Lemma \ref{not<}, the elements in $P$ are all pairwise incomparable, and so we have that $P{\sm}\{p\}\se \pt(\up a){\sm}\up p$, and so one of the desired inequalities is proven. For the other, suppose that $q$ is a prime above $a$ with $p\nleq q$. We have that $\bwe P{\sm}\{p\}\we p=\bwe P\leq q$, and by primality of $q$ this implies that $\bwe {\sm}\{p\}\leq q$. Indeed, then,  $\bwe P{\sm}\{p\}\leq\bwe \pt(\up a){\sm}\up p$.
\end{proof}

Let us gather the main results obtained in this subsection.
\begin{theorem}\label{important2}
For a frame $L$, the following are equivalent.
\begin{enumerate}
    \item The frame $L$ is totally spatial.
    \item The frame $\SLop$ is spatial.
    \item The map $\pt:\SL\ra \sob [\ca{P}(\pt(L))]$ is an isomorphism.
    \item Every Boolean sublocale of $L$ other then $\{1\}$ has at least a point.
    \item The frame $L$ is spatial and every element of $L$ other than $1$ has an essential prime.   
    \item The frame $L$ is spatial and for every element $a\neq 1$ there is a prime of the form $y\ra a$ for some $y\in L$. 
    \item Every element $a\in L$ is the meet of its essential primes.
    \item The frame $L$ is spatial and every closed subspace of $\pt(L)$ has a discrete dense subspace.
\end{enumerate}
\end{theorem}
\begin{proof}
The equivalence between (1) and (2) is given by Corollary \ref{totspaassemblyspa}. The equivalence between (1) and (3) follows from commutativity of the main diagram of this section. Items (1) and (4) are equivalent by Proposition \ref{enoughonly1point}. Items (1) and (5) are equivalent by Corollary \ref{essentialimportant}. Item (5) and (6) are equivalent because in a spatial frame all elements are the meet of the primes above it, and so by Lemma \ref{essentialiffinboolean}, for a spatial frame $L$ and an element $a\in L$ other than $1$, for a prime $p$ to be an essential prime of $a$ is equivalent to it being in $\bl(a)$, which is equivalent to it being of the form $y\ra a$ for some $y\in L$ by item (1) of Lemma \ref{booleanfacts}. Lemma \ref{meetofess} establishes that (1) and (7) are equivalent. Item (1) implies item (8) by Lemma \ref{essentialimplidds}, and item (8) implies item (1) by Proposition \ref{ddsimpliests}.
\end{proof}

\subsection{When do the sublocales perfectly represent the subspaces?}
This subsection is devoted to exploring the stronger property of the assembly of a spatial frame being Boolean. We give a new proof of the result established in \cite{simmons80} that these are exactly the scattered spaces. We will analyze scatteredness in terms of prime elements of frames of opens. A space is called \textit{scattered} if all its nonvoid subspaces contain an isolated point. The next results are aimed at characterizing scattered spaces in terms of prime elements. Recall that we say that an element $m\in M$ is an essential element of the meet $\bwe M$ if $\bwe M\neq \bwe (M{\sm}\{m\})$, and that this is equivalent to $\bwe (M{\sm}\{m\})\nleq m$.

\begin{lemma}\label{otherscatter}
A space $X$ is scattered if and only if every subspace of $X$ has a discrete dense subspace. 
\end{lemma}
\begin{proof}
Suppose that the space $X$ is scattered. Let $Y\se X$ be a subspace. Let $\{U_i:i\in I\}$ be the collection of opens of $X$ such that they intersect $Y$. Let $x_i$ be the isolated point of $U_i\cap Y$. The space $\{x_i:i\in I\}$ is a discrete dense subspace of $Y$: on the one hand it intersects every $U_i$, and on the other hand it is also discrete, since every point is isolated. Conversely, if every subspace of $X$ has a discrete dense subspace, so do all nonempty subspaces. By discreteness, for any nonempty subspace $Y$ its discrete dense subspace must be nonempty, as it intersects $Y$ itself, and so $Y$ must have at least one isolated point.
\end{proof}

\begin{lemma}\label{primes}
For every frame $L$ the space $\pt(L)$ is scattered if and only if for every collection $P\se \pt(L)$ the meet $\bwe P$ has an essential prime.
\end{lemma}
\begin{proof}
This follows from Lemma \ref{firstprime}.
\end{proof}

A frame is totally spatial if and only if $L$ is both spatial and such that every element $a=\bwe \pt(\up a)$ has a prime $p$ which is essential in the meet $\bwe \{q\in \pt(\up a):p\nleq q\mb{ or }p=q\}$. We now introduce a notion stronger than essentiality. For a frame $L$ and an element $a\in L$ which is the meet of the primes above it, we say that $p$ is an \tc{absolutely essential} prime of $a$ if and only if both $a\leq p$ and whenever $a=\bwe P$ for a collection of primes we have $p\in P$. We chose to call this property ``absolute" essentiality to express the contrast with classical essentiality, in which is relative to a particular way of writing an element as a meet of primes. Let us see various ways of characterizing this notion.

\begin{lemma}\label{particularmeet8}
Suppose that $L$ is a frame and $a\in L$ is a meet of primes. The following are equivalent for a prime $p\in \pt(\up a)$.
\begin{enumerate}
    \item The prime $p$ is an absolutely essential prime of $a$.
    \item The prime $p$ is an essential prime of the meet $\bwe \pt(\up a)$.
    \item The prime $p$ is an essential prime of $a$ and it is weakly covered.
\end{enumerate}
\end{lemma}
\begin{proof}
It is clear that (1) implies (2). Suppose, now, that (2) holds. A fortiori, we must have that $\bwe \pt(\up a){\sm}\up p\nleq a$, and so the prime $p$ is essential for $a$. Now suppose towards contradiction that $p$ is not weakly covered. We then have that $\bwe \{q\in \pt(\up a): p\neq q\}=\bwe \{q\in \pt(\up a):p<q\}\we \bwe \{q\in \pt(\up a):p\nleq q\}=p\we \bwe \pt(\up a){\sm}\up p\leq a$. We have contradicted (2). Suppose that $p$ is an essential prime of $a$ and that it is weakly covered. Let $P\se \pt(L)$ be a family of primes with $\bwe P=a$. Then, $P{\sm}\{p\}\se \pt(\up a){\sm}\{p\}$ and so $\bwe \pt(\up a){\sm}\{p\}\leq \bwe P{\sm}\{p\}$. Since $p$ is weakly covered we have $\bwe \pt(\up p){\sm}\{p\}\nleq p$, and since it is an essential prime of $a$ also $\bwe \pt(\up a){\sm}\up p\nleq p$. Combining these two inequalities, we deduce by primality of $p$ that $\bwe \pt(\up p){\sm}\{p\}\we \bwe \pt(\up a){\sm}\up p=\bwe \pt(\up a){\sm}\{p\}\nleq p$. We deduce that $\bwe P{\sm}\{p\}\nleq p$. If we had $p\notin P$, we would have that $\bwe P{\sm}\{p\}=\bwe P\leq a\leq p$, contradicting our last deduction. Then, $p\in P$.
\end{proof}
\begin{corollary}\label{essabscov}
For a spatial frame $L$ we have that every element other then $1$ has an absolutely essential prime if and only if every element other than $1$ has an essential prime and every prime is weakly covered.
\end{corollary}
\begin{proof}
It follows directly from Lemma \ref{particularmeet8} that if all elements other than $1$ have an essential prime and every prime is weakly covered then every element has an absolutely essential prime. For the converse, suppose that $L$ is spatial, and that every element other than $1$ has an absolutely essential prime. It is clear that every element has an essential prime. To see that every prime is covered, suppose that we have $P\se \pt(L)$ with $p=\bwe P$ a prime element. We claim that $p$ is an absolutely essential prime of itself: by assumption $p$ has at least one absolutely essential prime, and this prime must be contained in $\{p\}$ as $\bwe \{p\}=p$. As $p$ is an absolutely essential prime of itself, we must have $p\in P$, and so $p$ is weakly covered.
\end{proof}

\begin{proposition}\label{everyhasabsolutely}
For any frame $L$, the space $\pt(L)$ is scattered if and only if every element of $L$ other than $1$ which is a meet of primes has an absolutely essential prime.
\end{proposition}
\begin{proof}
Suppose that $L$ is a frame in which every element of $L$ which is a nonempty meet of primes has an absolutely essential prime. For a nonempty set $P\se \pt(L)$ we have that the element $\bwe P$, then, has an absolutely essential prime. By Lemma \ref{primes}, then, the space $P$ has an isolated point, and so the space $\pt(L)$ is scattered. Conversely, suppose that $L$ is a frame such that $\pt(L)$ is scattered. Suppose also that $a\in L$ is an element other than $1$ which is the meet of the primes above it. Consider the space $\pt(\up a)$. This must have an isolated point, say $p$, and by Lemma \ref{primes} this means that the prime $p$ is essential for the meet $\bwe \pt(\up a)$. By Lemma \ref{particularmeet8}, this means that $p$ is an absolutely essential prime of $a$.
\end{proof}

\begin{proposition}\label{allweaklytotspatial}
For any frame $L$ we have that it is spatial and $\pt(L)$ is scattered if and only if it is totally spatial and all its primes are weakly covered.
\end{proposition}
\begin{proof}
Suppose that a frame $L$ is spatial and that $\pt(L)$ is scattered. Then, by Proposition \ref{everyhasabsolutely}, every element of $L$ other than $1$ has an absolutely essential prime. By Corollary \ref{essabscov}, then, every element other than $1$ has an essential prime and every prime of $L$ is weakly covered. As $L$ is a spatial frame in which every element other than $1$ has an essential prime, it is totally spatial, by Theorem \ref{important2}. Conversely, suppose that $L$ is totally spatial and that all primes of $L$ are weakly covered. By Theorem \ref{important2}, the frame $L$ is spatial, and all elements of $L$ other then $1$ have an essential prime. Then, by Corollary \ref{essabscov}, every element other than $1$ has an absolutely essential prime. By Proposition \ref{everyhasabsolutely}, the space $\pt(L)$ is scattered.
\end{proof}

There is another way in which we may see that a spectrum $\pt(L)$ being scattered is a condition stronger then $L$ being totally spatial. Recall that for a frame $L$ and an element $a\in L$ which is a meet of primes a prime $p$ is essential if and only if it is an isolated point of $\pt(\up a){\sm}\up p$. We have then deduced that a frame $L$ is totally spatial if and only if every closed subspace of $\pt(L)$ has an isolated point. Strengthening this condition to $p$ being an absolutely essential prime of $a$ takes us to the following line of reasoning. 

\begin{lemma}\label{absisolated}
For a frame $L$ and an element $a\in L$, if $a$ is a meet of primes we have that $p$ is an absolutely essential prime of $a$ if and only if it is an isolated point of $\pt(\up a)$.
\end{lemma}
\begin{proof}
This follows from the equivalence of the first two items in Lemma \ref{particularmeet8}, and by Lemma \ref{firstprime}.
\end{proof}

For a frame $L$, we refer as $\mf{AbsEss}(a)$ to the collection of all absolutely essential primes for each element $a\in L$.  

\begin{lemma}\label{ifddsthenabsess}
For a frame $L$ and an element $a\in L$, is $\pt(\up a)$ has a discrete dense subspace then it is $\mf{AbsEss}(a)$.
\end{lemma}
\begin{proof}
Suppose that $L$ is a frame and that $a\in L$. Suppose that the space $\pt(\up a)\se \pt(L)$ has a discrete dense subspace; let this be $D$. By discreteness, all points of $D$ must be isolated in $\pt(\up a)$, and we know, by Lemma \ref{absisolated}, that this means they are all absolutely essential primes of $a$, so that we obtain $D\se \mf{AbsEss}(a)$. For the reverse set inclusion, we notice that any discrete dense subspace of a space must consist also of all isolated points of the space, or the singleton not in the subspace would be a nonempty open disjoint from it.  
\end{proof}

\begin{proposition}\label{meetofabsess}
For a frame $L$ we have that $L$ is spatial and $\pt(L)$ scattered if and only if every element $a\in L$ is the meet of its absolutely essential primes.
\end{proposition}
\begin{proof}
Suppose that $L$ is a frame and that $a\in L$. First, suppose that $a=\bwe \mf{AbsEss}(a)$. We claim that this implies that $\mf{AbsEss}(a)\se \pt(\up a)$ is a discrete dense subspace. The space is discrete, since by Lemma \ref{absisolated} it consists of isolated points. To show that it is dense, we show that any open subset of $\pt(\up a)$ disjoint from it is empty. Suppose, then, that $\{q\in \pt(\up a):b\nleq q\}$ is disjoint from $\mf{AbsEss}(a)$. We then have $b\leq \bwe \mf{AbsEss}(a)=a$ and so we cannot possibly have a prime $p$ above $a$ with $b\nleq p$. Then,this open set is empty and so $\mf{AbsEss}(a)$ is dense. For the converse, suppose that $L$ is spatial and that $a\in L$, and that $\mf{AbsEss}(a)$ is a discrete dense subspace of $\pt(\up a)$ (see Lemma \ref{meetofabsess}). Let us show that $\bwe \mf{AbsEss}(a)\leq a$. By spatiality, it suffices to show that for every prime $q$ we have that $a\leq q$ implies $\bwe \mf{AbsEss}(a)\leq q$. Suppose, then, that $q\in \pt(\up a)$. Consider the open set $\{r\in \pt(\up a):\bwe \mf{AbsEss}(a)\nleq r\}$. This open set does not intersect $\mf{AbsEss}(a)$. Then, it must be empty. In particular, it cannot contain $q$, that is, we must have $\bwe \mf{AbsEss}(a)\leq q$. Then, indeed $\bwe \mf{AbsEss}(a)\leq a$.
\end{proof}

Let us gather the results obtained in this subsection in a final theorem.

\begin{theorem}
For a frame $L$, the following are equivalent.
\begin{enumerate}
    \item The frame $L$ is spatial and the space $\pt(L)$ is scattered.
    \item The frame $L$ is totally spatial and every prime of $L$ is weakly covered.
    \item The map $\pt:\SL\ra \ca{P}(\pt(L))$ is an isomorphism.
    \item There exists an isomorphism $\SL\cong \ca{P}(\pt(L))$.
    \item The coframe $\SL$ is Boolean and $L$ is spatial.
    
    \item The frame $L$ is spatial and all elements other than $1$ have an absolutely essential prime.
    \item All elements of $L$ are the meet of their absolutely essential primes.
    \item The frame $L$ is spatial and every subspace of $\pt(L)$ has a discrete dense subspace.
\end{enumerate}
\end{theorem}
\begin{proof}
The equivalence between (1) and (2) is established by Proposition \ref{allweaklytotspatial}. Items (2) and (3) are equivalent because of Theorems \ref{important1} and \ref{important2}: the frame $L$ is totally spatial and such that all its primes are weakly covered if and only if both the vertical arrows of the main diagram of this section are isomorphisms. It is clear that (3) implies (4). Let us show that (4) implies (3). Suppose, then, that there is an isomorphism $\SL\ra \ca{P}(\pt(L))$. Since then $\SLop$ is a Boolean algebra, this isomorphism must restrict to a bijection between the two sets of atoms. But, since these two isomorphic Boolean algebras are both powersets, any bijection between their sets of atoms extends to an order isomorphism between them. Recall also that in $\SL$ the meet primes are the two-elements sublocales, that is, the Boolean sublocales of the form $\bl(p)$ for some $p\in \pt(L)$. Additionally, since $\SL$ is Boolean, the meet primes also coincide with the atoms (Lemma \ref{atomsofba}). The map $\pt:\SL\ra \ca{P}(\pt(L))$, indeed, is a bijection between the two sets of atoms since it acts as $\bl(p)\mapsto p$. Then, it must be an isomorphism. If (4) holds, clearly $\SL$ is Boolean. Additionally, $\SLop$ is also spatial, and this implies that $L$ is totally spatial, in particular it is spatial. Then, we have (5). Let us show that (5) implies (4). Suppose, then, that $L$ is spatial and $\SL$ is Boolean. Any spatial Boolean frame is isomorphic to the powerset of its meets primes, by Lemmas \ref{atomsofba} and \ref{spatialba}. Then, if $\SL$ is Boolean it is isomorphic to $\ca{P}(\pt(\mf{A}(L)))$. Since $\mf{A}(L)$ has as many points as $L$, this is isomorphic to $\ca{P}(\pt(L))$. The equivalence between (1) and (6) is stated by Proposition \ref{everyhasabsolutely}. Items (1) and (7) are equivalent by Proposition \ref{meetofabsess}. Items (1) and (8) are equivalent by Lemma \ref{otherscatter}.       
\end{proof}

\subsection{Spaces with frames of opens which are coframes}

In \cite{Avila19} the authors notice in the introduction that for an Alexandroff space $X$ we have that $X$ being sober already implies that the space is scattered (this is phrased as follows: for a poset $P$ we have that the frame of nuclei of the frame of upsets $\ca{U}(P)$ is Boolean if and only if the space $(P,\ca{U}(P))$ is sober). Here, we see that this result can be generalized to every space $X$ such that $\Om(X)$ is a coframe. We also will show that both these conditions are equivalent to all the primes of $\Om(X)$ being completely prime. Not all frames $\Om(X)$ which are coframes come from Alexandroff spaces. In particular, there are spatial frames which are coframes without being completely distributive (and all lattices of opens of Alexandroff spaces are completely distributive). The interactions between various kinds of distributivity on frames are explored in great detail in \cite{erne09}.

\begin{lemma}\label{spatialcompletelyprime}
If $L$ is a spatial frame, and $a\in L$ is an element such that $\bwe P\leq a$ implies that $p\leq a$ for some $p\in P$ for any collection of primes $P$, then the element $a$ is completely prime.
\end{lemma}
\begin{proof}
Suppose that $L$ is a spatial frame. Suppose that $a\in L$ is such that whenever $P\se \pt(L)$ and $\bwe P\leq a$ implies $p\leq a$ for some $p\in P$. Suppose that $\bwe A\leq a$ for $A\se L$. Since $L$ is spatial, this implies $\bwe_{x\in A}\bwe \pt(\up x)\leq a$. By assumption on $a$, we have that there is some $x\in A$ and some $p\in \pt(\up x)$ such that $p\leq a$. Then, we also have $x\leq a$. 
\end{proof}
\begin{lemma}\label{bigmeetisprime}
For a lattice $L$, codirected meets of primes are prime.
\end{lemma}
\begin{proof}
Suppose that $L$ is a frame, and that $p_i\in \pt(L)$ is a collection of primes such that for each $j,k\in I$ there is $i\in I$ with $p_i\leq p_j\we p_k$. Suppose that $x\we y\leq \bwe _i p_i$. Towards contradiction, suppose that $x\nleq \bwe _i p_i$ and $y\nleq \bwe_i p_i$. In particular, let $x\nleq p_x$ and $y\nleq p_y$. There is $p_{xy}\in \{p_i:i\in I\}$ such that $p_{xy}\leq p_x\we p_y$. By assumption, $x\we y\leq p_{xy}$. By primality of $p_{xy}$, either $x\leq p_{xy}\leq p_x\we p_y$, or $y\leq p_{xy}\leq p_x\we p_y$. In the first case we contradict that $x\nleq p_x$, and in the second that $y\nleq p_y$. Then, we must have that either $x \leq \bwe _i p_i$ or $y\leq \bwe_i p_i$.
\end{proof}
\begin{lemma}\label{noinfdeschain}
If $L$ is a frame and all primes of $L$ are completely prime, there can be no infinitely descending chains of primes in $L$.
\end{lemma}
\begin{proof}
If we have a frame $L$ and an infinitely descending chain of primes $p_i$, we must have that $\bwe_i p_i\notin \{p_i:i\in I\}$, or $\{p_i:i\in I\}$ would have a minimum, contradicting that it is infinitely descending. Furthermore, since $\{p_i:i\in I\}$ is a chain, it is codirected. Then, the meet $\bwe_i p_i$ is prime, by Lemma \ref{bigmeetisprime}. So, there is a prime $\bwe_i p_i$ such that $\bwe _i p_i\leq \bwe _i p_i$ but $p_i\nleq \bwe _i p_i$ for every $i\in I$, that is, $\bwe_i p_i$ is not completely prime.
\end{proof}

\begin{proposition}
Suppose that $L$ is a spatial frame and a coframe. The following are equivalent.
\begin{enumerate}
    \item The space $\pt(L)$ is $T_D$.
    \item All primes of $L$ are completely prime.
    \item The space $\pt(L)$ is scattered.
\end{enumerate}
\end{proposition}
\begin{proof}
Suppose that $L$ is a spatial frame which is a coframe. First, suppose that (1) holds. Suppose that $p\in \pt(L)$. By Lemma \ref{spatialcompletelyprime}, it suffices to show that whenever $P\se \pt(L)$ and $\bwe P\leq p$ we have that $q\leq p$ for some $q\in P$. Suppose, then, that the antecedent holds. Then, $\bwe (P\ve p)=p$, where we have used the coframe distributivity law. Since $\pt(L)$ is $T_D$, all primes of $L$ are weakly covered (see Theorem \ref{important1}). Then, we must have $q\ve p=p$ for some $q\in P$, that is, $q\leq p$. Let us show that (2) implies (3). Suppose that (2) holds, and that $P\se \pt(L)$. To show that (3) holds we need to show that there is an essential prime of the meet $\bwe P$. By Lemma \ref{noinfdeschain}, we have that there are no infinitely descending chains in $P$, and so there must be a minimal element, say $p\in P$. We then have that $q\nleq p$ for every $q\in P{\sm}\{p\}$. Since $p$ is completely prime, this also implies that $\bwe P{\sm}\{p\}\nleq p$, that is, $p$ is an essential element of the meet $\bwe P$. Finally, let us show that (3) implies (1). Suppose that $\pt(L)$ is scattered and that $P\se \pt(L)$ is a collection of primes such that $\bwe P$ is prime. Suppose towards contradiction that $p\notin P$. Then, we have that $\bwe P\we p$ is a meet of primes without essential prime. This contradicts $\pt(L)$ being scattered.  
\end{proof}

\bibliographystyle{elsarticle-harv}
\bibliography{bibliorelation}

\begin{thebibliography}{14}
\expandafter\ifx\csname natexlab\endcsname\relax\def\natexlab#1{#1}\fi
\providecommand{\url}[1]{\texttt{#1}}
\providecommand{\href}[2]{#2}
\providecommand{\path}[1]{#1}
\providecommand{\DOIprefix}{doi:}
\providecommand{\ArXivprefix}{arXiv:}
\providecommand{\URLprefix}{URL: }
\providecommand{\Pubmedprefix}{pmid:}
\providecommand{\doi}[1]{\href{http://dx.doi.org/#1}{\path{#1}}}
\providecommand{\Pubmed}[1]{\href{pmid:#1}{\path{#1}}}
\providecommand{\bibinfo}[2]{#2}
\ifx\xfnm\relax \def\xfnm[#1]{\unskip,\space#1}\fi
\bibitem[{Aull and Thron(1963)}]{Aull62}
\bibinfo{author}{Aull, C.E.}, \bibinfo{author}{Thron, W.J.},
  \bibinfo{year}{1963}.
\newblock \bibinfo{title}{{Separation Axioms Between $T_0$ and $T_1$}}.
\newblock \bibinfo{journal}{Indagationes Mathematicae} ,
  \bibinfo{pages}{26--37}.
\bibitem[{Banaschewski and Pultr(2010)}]{banaschewskitd}
\bibinfo{author}{Banaschewski, B.}, \bibinfo{author}{Pultr, A.},
  \bibinfo{year}{2010}.
\newblock \bibinfo{title}{{Pointfree Aspects of the $T_D$ Axiom of Classical
  Topology}}.
\newblock \bibinfo{journal}{Quaestiones Mathematicae} \bibinfo{volume}{33},
  \bibinfo{pages}{369--385}.
\bibitem[{Banaschewski and Pultr(2014)}]{banaschewski15}
\bibinfo{author}{Banaschewski, B.}, \bibinfo{author}{Pultr, A.},
  \bibinfo{year}{2014}.
\newblock \bibinfo{title}{On covered prime elements and complete homomorphisms
  of frames}.
\newblock \bibinfo{journal}{Quaestiones Mathematicae} \bibinfo{volume}{37},
  \bibinfo{pages}{451--454}.
\bibitem[{Erné(2009)}]{erne09}
\bibinfo{author}{Erné, M.}, \bibinfo{year}{2009}.
\newblock \bibinfo{title}{Infinite distributive laws versus local connectedness
  and compactness properties}.
\newblock \bibinfo{journal}{Topology and its Applications}
  \bibinfo{volume}{156}, \bibinfo{pages}{2054--2069}.
\bibitem[{Isbell(1972)}]{isbell72}
\bibinfo{author}{Isbell, J.}, \bibinfo{year}{1972}.
\newblock \bibinfo{title}{Atomless parts of spaces}.
\newblock \bibinfo{journal}{Mathematica Scandinavica} \bibinfo{volume}{31},
  \bibinfo{pages}{5--32}.
\bibitem[{Isbell(1991)}]{ISBELL91}
\bibinfo{author}{Isbell, J.}, \bibinfo{year}{1991}.
\newblock \bibinfo{title}{On dissolute spaces}.
\newblock \bibinfo{journal}{Topology and its Applications}
  \bibinfo{volume}{40}, \bibinfo{pages}{63 -- 70}.
\bibitem[{Johnstone(1982)}]{johnstone82}
\bibinfo{author}{Johnstone, P.T.}, \bibinfo{year}{1982}.
\newblock \bibinfo{title}{Stone Spaces}. volume~\bibinfo{volume}{3} of
  \textit{\bibinfo{series}{Cambridge Studies in Advanced Mathematics}}.
\newblock \bibinfo{publisher}{Cambridge University Press}.
\bibitem[{Niefield and Rosenthal(1987)}]{niefield87}
\bibinfo{author}{Niefield, S.}, \bibinfo{author}{Rosenthal, K.},
  \bibinfo{year}{1987}.
\newblock \bibinfo{title}{Spatial sublocales and essential primes}.
\newblock \bibinfo{journal}{Topology and its Applications}
  \bibinfo{volume}{26}, \bibinfo{pages}{263 -- 269}.
\bibitem[{Picado and Pultr(2011)}]{picadopultr2011frames}
\bibinfo{author}{Picado, J.}, \bibinfo{author}{Pultr, A.},
  \bibinfo{year}{2011}.
\newblock \bibinfo{title}{Frames and Locales: Topology without points}.
\newblock \bibinfo{publisher}{Springer-Birkh{\"a}user Basel}.
\bibitem[{Picado and Pultr(2019)}]{picado19}
\bibinfo{author}{Picado, J.}, \bibinfo{author}{Pultr, A.},
  \bibinfo{year}{2019}.
\newblock \bibinfo{title}{{Axiom $T_D$ and the Simmons sublocale theorem}}.
\newblock \bibinfo{type}{{Pr\'{e}-Publica\c{c}\~{o}es do Departamento de
  Matem\'{a}tica}} \bibinfo{number}{18-48}. Universidade de Coimbra.
\bibitem[{Pultr and Tozzi(1994)}]{Pultr94}
\bibinfo{author}{Pultr, A.}, \bibinfo{author}{Tozzi, A.}, \bibinfo{year}{1994}.
\newblock \bibinfo{title}{Separation axioms and frame representation of some
  topological facts}.
\newblock \bibinfo{journal}{Applied Categorical Structures}
  \bibinfo{volume}{2}, \bibinfo{pages}{107--118}.
\bibitem[{Simmons(1980)}]{simmons80}
\bibinfo{author}{Simmons, H.}, \bibinfo{year}{1980}.
\newblock \bibinfo{title}{{Spaces with Boolean assemblies}}.
\bibitem[{Ávila et~al.(2019a)Ávila, Bezhanishvili, Morandi and
  Zaldívar}]{Avila19}
\bibinfo{author}{Ávila, F.}, \bibinfo{author}{Bezhanishvili, G.},
  \bibinfo{author}{Morandi, P.}, \bibinfo{author}{Zaldívar, A.},
  \bibinfo{year}{2019}a.
\newblock \bibinfo{title}{{The Frame of Nuclei of an Alexandroff Space}}.
\newblock \bibinfo{note}{ArXiv:1906.03640}.
\bibitem[{Ávila et~al.(2019b)Ávila, Bezhanishvili, Morandi and
  Zaldívar}]{avila2019frame}
\bibinfo{author}{Ávila, F.}, \bibinfo{author}{Bezhanishvili, G.},
  \bibinfo{author}{Morandi, P.}, \bibinfo{author}{Zaldívar, A.},
  \bibinfo{year}{2019}b.
\newblock \bibinfo{title}{When is the frame of nuclei spatial: A new approach}.
\newblock \bibinfo{note}{{arXiv:1906.03636}}.

\end{thebibliography}
\end{document}